\documentclass[a4paper,11pt,reqno]{amsart}
\usepackage{amsthm}
\usepackage{amsmath}
\usepackage{enumerate}
\usepackage{amsfonts}
\usepackage{amssymb}
\usepackage{fullpage}
\usepackage{amsmath,amscd}
\usepackage{stmaryrd}
\usepackage{graphicx}
\usepackage{tikz-cd}
\usepackage{array}
\usepackage{amsmath,amssymb,amsfonts,dsfont}
\usepackage[utf8]{inputenc}
\usepackage[english]{babel}
\usepackage[T1]{fontenc}
\usepackage{graphicx}
\usepackage{float}
\usepackage{pdfpages}
\usepackage[a4paper, margin = 3cm, bottom = 3cm]{geometry}
\usepackage{ifpdf}
\usepackage{marginnote}
\usepackage{hyperref}	

\usetikzlibrary{calc}
\usetikzlibrary{decorations.pathreplacing,decorations.markings,decorations.pathmorphing}
\usetikzlibrary{positioning,arrows,patterns}
\usetikzlibrary{cd}
\usetikzlibrary{intersections}
\usetikzlibrary{arrows}

\hypersetup{pdfborder=0 0 0,
	    colorlinks=true,
	    citecolor=black,
	    linkcolor=blue,
	    urlcolor=red,
	    pdfauthor={Guillaume Tahar}
	   }

\DeclareMathOperator{\lcm}{lcm}
\DeclareMathOperator{\res}{res}
\newtheorem{thm}{Theorem}[section]
\newtheorem{cor}[thm]{Corollary}
\newtheorem{prop}[thm]{Proposition}
\newtheorem{lem}[thm]{Lemma}

\theoremstyle{definition}
\newtheorem{defn}[thm]{Definition}
\theoremstyle{remark}
\newtheorem{rem}[thm]{Remark}
\theoremstyle{definition}

\theoremstyle{definition}
\newtheorem{ex}[thm]{Example}
\theoremstyle{definition}

\numberwithin{equation}{section}
\title{Isoresidual fibration and resonance arrangements}

\author{Quentin Gendron}
\address[Quentin Gendron]{Centro de Ciencas Matematicas - UNAM Campus Morelia
Antigua Carretera a Pátzcuaro, 8701
Col. Ex Hacienda San José de la Huerta
Morelia, Michoacán, México
C.P. 58089}
\curraddr{Centro de Investigacion en Matematicas, Guanjuato, Gto.,  AP 402, CP 36000, México}

\email{quentin.gendron@cimat.mx}

\author{Guillaume Tahar}
\address[Guillaume Tahar]{Faculty of Mathematics and Computer Science, Weizmann Institute of Science,
Rehovot, 7610001, Israel}
\email{tahar.guillaume@weizmann.ac.il}

\date{February 25, 2022}
\begin{document}

\begin{abstract}
The stratum $\mathcal{H}(a,-b_{1},\dots,-b_{p})$ of meromorphic $1$-forms with a zero of order $a$ and poles of orders $b_{1},\dots,b_{p}$ on the Riemann sphere has a map, the isoresidual fibration, defined by assigning to any differential its residues at the poles. We show that above the complement of a hyperplane arrangement, the resonance arrangement, the isoresidual fibration is an unramified cover of degree $\frac{a!}{(a+2-p)!}$. Moreover, the monodromy of the fibration is computed for strata with at most three poles and a system of generators and relations is given for all strata.
These results are obtained by associating to special differentials of the strata a tree, and by studying the relationship between the geometric properties of the differentials and the combinatorial properties of these trees.
\end{abstract}
\maketitle
\setcounter{tocdepth}{1}
\tableofcontents

\section{Introduction}

For any $g\geq 0$ and any partition $(a_{1},\dots,a_{n},-b_{1},\dots,-b_{p})$ of $2g-2$, we denote by $\mathcal{H}(a_{1},\dots,a_{n},-b_{1},\dots,-b_{p})$ the {\em stratum} of  meromorphic abelian differentials $1$-forms whose zeroes are of order $a_{i}$ and poles are of order~$b_{j}$. For any stratum with  $p\geq1$, we define the \textit{residual space} $\mathcal{R}_{p}$ to be the complex vector space formed by the vectors $(\lambda_{1},\dots,\lambda_{p})$ such that $\sum_{j=1}^{p} \lambda_{j} = 0$ and the \textit{residual map}
\[\res\colon \mathcal{H}(a_{1},\dots,a_{n},-b_{1},\dots,-b_{p}) \to \mathcal{R}_{p} : \omega \mapsto (\lambda_{1},\dots,\lambda_{p})\]
that assigns to each differential $\omega$ of a given stratum of meromorphic $1$-forms the sequence of its residues $\lambda_{i}$ at its poles~$p_{i}$. This map defines the \textit{isoresidual fibration} of the stratum. \newline

The residual map is still highly mysterious for $p\geq2$: only its image is known by results of \cite{GT}. In the case of the Riemann sphere  $\mathbb{CP}^{1}$,  any meromorphic $1$-form $\omega$ can be  written $\omega = \frac{P(z)}{Q(z)}dz$ where $P$ and $Q$ are relatively prime polynomials in~$z$. Hence, computing the residue of this $1$-form amounts to compute the value of a complicated determinant in terms of the locations of the zeroes and the poles. \newline

The strata of genus zero are the product of $\mathbb{C}^{\ast}$ with the configuration space of a finite set of points on the sphere. Since projective automorphisms allow to locate three points in $0$, $1$ and $\infty$, the complex dimension of a stratum is $n+p+1-3=n+p-2$. Fixing the residue of the differentials at each pole defines $p-1$ independent such equations in the stratum. Hence the residual map is finite if and only if $n=1$.\newline

In this paper, we provide a study of the isoresidual fibration for the strata of differential $1$-forms on the Riemann sphere with a unique zero. This case is especially interesting since the fibers of the fibration are  discrete and the isoresidual fibration is a branched cover of the stratum to the residual space.\newline

\paragraph{\bf Main Results.}

For any partition $(a,-b_{1},\dots,-b_{p})$ of $-2$ with $p\geq2$, we consider the stratum $\mathcal{H}=\mathcal{H}(a,-b_{1},\dots,-b_{p})$ of meromorphic $1$-forms. In the whole article the poles are labeled. The residual space $\mathcal{R}_{p}$ is a complex vector space of dimension~$p-1$.

\begin{defn}\label{def:resonanceintro}
Let $I$ be a nonempty proper subset of $\{1,\dots,p\}$, the \textit{resonance hyperplane} $A_{I}$ is the space defined by the equation $\sum_{j \in I} \lambda_{j} = 0$  in the residual space $\mathcal{R}_{p}$. \newline
The union of all the resonance hyperplanes is the \textit{resonance arrangement} $\mathcal{A}_{p} \subset \mathcal{R}_{p}$.
\end{defn}

The main theorem of this paper, proved in Section~\ref{sec:nonresonance}, is that the isoresidual cover is unramified over the complement of the resonance arrangement and the computation of its degree.

\begin{thm}\label{thm:numerofibre}
For every stratum of meromorphic $1$-forms $\mathcal{H}(a,-b_{1},\dots,-b_{p})$ on $\mathbb{CP}^{1}$, the isoresidual fibration is a non ramified cover of degree $\frac{a!}{(a+2-p)!}$ over $\mathcal{R}_{p} \setminus \mathcal{A}_{p}$.
\end{thm}

It is worth of noting that the degree does not depends on the orders of every individual poles but only on the number of poles and the sum of the orders. This is very surprising and we show that over the resonance arrangement $\mathcal{A}_{p}$ the number of elements in the fibers depends on the individual orders. Moreover there are strictly less elements in the fibers over the resonance arrangement. This comes from two phenomena, the first one is ramification and the second is that some elements correspond to singular $1$-forms. This is stated precisely in Section~\ref{sec:degentree} and used in Section~\ref{sec:comptereson} to compute the number of elements over some resonance hyperplanes. In particular, the number of fibers lying above exactly one hyperplane is given by the following proposition.
\begin{prop}\label{prop:degintro}
 The number of elements of an isoresidual fiber $\mathcal{F}$ that belongs to exactly one resonance hyperplane $A_{I}$ is $$\frac{a!}{(a+2-p)!}-\frac{d_{I}!\, (a-d_{I})!}{(d_{I}+1-c_{I})! \, (a+1-p+c_{I}-d_{I})!}\,,$$
 where $c_{I}$ is the cardinal of $I$ and $d_{I}=-1+\sum_{j \in I} b_{j}$ is  the \textit{resonance degree} of~$I$.
\end{prop}

Now, the resonance arrangement enjoys several pleasing properties: it is central, i.e. the origin of the vector space belongs to every hyperplane of the arrangement, and it is the complexification of a real arrangement, i.e. the coefficients of the equations of hyperplanes are real. However, as we show it in Lemma~\ref{lem:passimplice}, if $p \geq 4$, the resonance arrangement is not simplicial (it is not the complexification of a real hyperplane arrangement whose complement is a union of simplices). Hence, the fundamental group $W_{p}$ of $\mathcal{R}_{p} \setminus \mathcal{A}_{p}$ is still unknown in general.\newline

For any stratum $\mathcal{H}$ with $p$ poles, the monodromy of the isoresidual cover can be identified with a subgroup of the (permutation) group of automorphisms of any generic fiber and can be also identified with a quotient of the fundamental group $W_{p}$. For any resonance arrangement $\mathcal{A}_{p}$, there are infinitely many isoresidual covers (of arbitrarily high degree) on which $W_{p}$ acts by monodromy. As such, it provides an indirect way to study $W_{p}$.\newline

In Section~\ref{sec:monodromie}, we study the monodromy group of the isoresidual cover for strata of meromorphic $1$-forms. For every stratum, we describe a system of generators and relations. In the case of three poles, this allows us to obtain the following characterisation.
\begin{thm}\label{thm:monothree}
For any stratum $\mathcal{H}(a,-b_{1},-b_{2},-b_{3})$ of genus zero, the monodromy group of the isoresidual cover is isomorphic to:
\begin{enumerate}[(i)]
 \item the cyclic group $\mathbb{Z}_{a}$ if $b_{2}=b_{3}=1$;
 \item the exotic embedding of $\mathfrak{S}_{5}$ into $\mathfrak{S}_{6}$ if $b_{1}=2$ and $b_{2}=b_{3}=3$;
 \item the alternating group $\mathfrak{A}_{a}$ if $b_{1},b_{2} \geq 2$ and $b_{1},b_{2},b_{3}$ have the same parity;
 \item the symmetric group $\mathfrak{S}_{a}$ otherwise.
\end{enumerate}
\end{thm}
\smallskip
\par

\paragraph{\bf Ideas of proof.}
The problem is approached geometrically in the framework of translation surfaces, see \cite{Zo}. Indeed, the integration of differential $1$-forms defines a developing map of a geometric structure locally modelled on $\mathbb{C}$ up to translations. Any Riemann surface endowed with meromorphic $1$-form $\omega$ leads to a \textit{translation surface (with poles)}. Usually, the dictionary between complex analysis and flat geometry is used to study dynamical systems defined on translation surfaces (coming from rational billiards) with tools from algebraic geometry. An excellent account of such use is given by~\cite{Moe}. In this paper, we do exactly the opposite. We translate a purely algebraic problem into a question about geometry of translation surfaces. The classification of the connected components of the strata in \cite{KZ} is one of the most striking result proceeding in this direction.\newline

More precisely, the proof of Theorem \ref{thm:numerofibre} contains several steps. The first is to show that any fiber of the isoresidual fibration above $\mathcal{R}_{p} \setminus \mathcal{A}_{p}$ is in bijection with a fiber over a configuration of real residues. These translation surfaces have only real residues, and hence are especially easy to handle. For example, we show that their geometry is completely characterized by their residues and some combinatorial informations encompassed in the \textit{decorated tree} that we introduce in Section~\ref{sec:decorated}. The enumeration of the decorated trees, mainly by combinatorial methods, in Section~\ref{sec:nonresonance} completes the proof of Theorem~\ref{thm:numerofibre}.\newline

 The interpretation in terms of trees is especially highlighting for degeneration phenomena. Indeed, the edges of the trees correspond to closed loops whose length is determined by the residues of the $1$-form. If a resonance equation is satisfied by the configuration of residues, then the differential that should correspond to this tree may be singular. The closed loop corresponding to the edge shrinks and we get a degenerate object corresponding to an element of the WYSIWYG compactification as discussed in Section~\ref{sec:degentree}. Using this interpretation, we prove Proposition~\ref{prop:degintro} in Section~\ref{sec:comptereson}.\newline

 The loops around the hyperplanes of $\mathcal{A}_{p}$ may display some nontrivial monodromy in the fibers. The identification of the degenerated elements of the isoresidual fiber with degenerated decorated trees allows us to understand combinatorially the monodromy of the fibration around each resonance hyperplane in terms of surgery of these trees. We provide the decomposition into cycles of the permutation induced on a fiber by the monodromy action of a loop around a hyperplane resonance in Propositions~\ref{prop:monodromieloop} and~\ref{prop:monodromieloopbis}. Combining these results with some classical results on permutation groups we prove Theorem~\ref{thm:monothree}.\newline

\paragraph{\bf  Related works and possible developments.}

Despite the importance of the isoresidual fibration, only few results are known about it. The image of the map has been computed in \cite{GT}. The present paper gives the first general results on the geometry of the fibration. In the spirit of Theorem~\ref{thm:numerofibre}, the first non trivial results in genus $0$ about the number of fibers above some special configurations of residues has been obtained in \cite{CC}. Our setting allows us to give an alternative proof of their results in this direction, see Proposition~\ref{prop:trivial}. In genus higher than zero, it seems that only few special cases in genus~$1$ are known, see \cite{CG}. \newline 

The study of the monodromy in enumerative problems is a classical problem. As in \cite{Ha} we compute it by "drawing arcs in the parameter space". The monodromy that we obtain share similarities with the ones appearing in that paper, but we show in Corollary~\ref{cor:alternate} that surprisingly the monodromy for a strata with more than $6$ poles is contained in the alternating group.\newline

A deeper study of the monodromy action of the fundamental group of resonance arrangements will be carried out in a subsequent work. For differential $1$-forms in genus zero with two zeroes, the isoresidual fibers are complex curves with cusps corresponding to degenerations where the two zeroes collide. The monodromy action for such strata is then understood as a representation of the group $W_{p}$ into the mapping class group, the action on the cusps being inherited from the case with one zero. This places our study in the classical field of representations of braid groups (and their relatives) into mapping class groups of manifolds.\newline

The importance of the isoresidual fibration comes in particular from unexpected connections with several other fields in mathematics. The meromorphic differentials with real periods have been extensively used under the denomination of Real-Normalized meromorphic differentials in \cite{GK1,GK2} in relation with integrable systems, like the Calogero-Moser system. Moreover, in the case of strata of differentials of genus zero with only one zero, the decorated trees that we introduce can be interpreted as a generalization of the Douady-Sentenac invariant introduced in \cite{DS} and used in \cite{CR} to classify deformations of polynomial vector fields in the complex plane.
Finally, the resonance arrangement is related to the double Hurwitz numbers, see \cite{CJM}, and many optimization problems, see \cite{KTT}.\newline

Finally, Dawei Chen and Miguel Prado informed us during the revision process that they reproved Theorem~\ref{thm:numerofibre} by other methods. Moreover, they could extend Proposition~\ref{prop:degintro} to the intersection of any number of resonance hyperplanes. See their work \cite{chpr}.\newline

\paragraph{\bf Organisation.}
The structure of the paper is the following:
\begin{itemize}
 \item In Section \ref{sec:merotrans}, we recall the background about translation surfaces and how they arise from $1$-forms. We discuss in particular the flat metric, saddle connections, and the contraction flow.
 \item In Section \ref{sec:transarbrees}, we introduce the resonance arrangement and the decorated trees associated to $1$-forms over real configurations of residues both above the nonresonant and the resonant locus.
 \item In Section \ref{sec:compter}, we enumerate the decorated trees in each strata above the nonresonance locus, proving Theorem~\ref{thm:numerofibre}, and above some resonant loci.
 \item In Section \ref{sec:monodromie}, we study the monodromy of the isoresidual fibration through its action on decorated trees, proving Theorem~\ref{thm:monothree}.
\end{itemize}

\section{From meromorphic $1$-forms to translation surfaces}
\label{sec:merotrans}

In this section, we recall the main concepts from flat geometry that we need. For more information, the reader can consult one of the very good introductory text including, but not restricted to, \cite{Ch,Wr,Zo}.

\subsection{Translation structures and local models}
\label{sec:locmod}

Let us consider a meromorphic differential $1$-form $\omega$ on a Riemann surface $X$. Such $\omega$ is locally of the form $f(z)dz$ where $f$ is a meromorphic function and $z$ a local coordinate. We denote by $\mathcal{H}(a_{1},\dots,a_{n},-b_{1},\dots,-b_{p})$ for $a_{i},b_{j}\geq1$ the \textit{stratum} that parametrizes to meromorphic $1$-forms with zeroes of orders $a_1,\dots,a_n$ and poles of order $b_1,\dots,b_p$ up to biholomorphism. The theorem of Riemann-Roch implies that $\sum_{i=1}^{n} a_{i} -\sum_{j=1}^{p} b_{j} = 2g-2$, where $g$ is the genus of $X$. In this paper, we focus on meromorphic $1$-forms on the Riemann sphere.\newline
We denote by $\Lambda$ the set of the zeroes of $\omega$ and by $\Delta$ the set of its poles. 
It should be noted that the poles are labeled in order to avoid complicated symmetry issues.\newline

Outside $\Lambda$ and $\Delta$, the integration of $\omega$ gives local charts to $\mathbb{C}$ whose transition maps are of the type $z \mapsto z+c$. The pair $(X,\omega)$ seen as a compact surface with such an atlas is called a \textit{translation surface}.
In a neighborhood of a zero of order $a>0$, the metric induced by $\omega$ admits a conical singularity of angle $(1+a)2\pi$, see \cite{Zo} for details.\newline

Before giving the local models for poles of the differential, we recall the following convention. The \textit{residue} of the differential at a pole is the period of the differential over a simple loop (positively oriented) around it. This convention differs from the usual one by a factor of $2i\pi$. Since our approach focuses on periods rather than coefficients, this convention is more suitable. In particular, in the case of strata $\mathcal{H}(a,-b_{1},\dots,-b_{p})$ of genus zero, if the residue at every pole is real, the period of any closed loop is real.\newline

In a neighborhood of a simple pole, the $1$-form $\omega$ is of the form $\dfrac{rdz}{z}$ where $2i\pi r$ is the residue at the pole. Its geometric interpretation is a semi-infinite cylinder whose waist curves have a period equal to residue $2i\pi r$, see \cite{Bo}. In particular, the residue is the only local invariant of a simple pole.\newline

A {\em flat cone of type $(a,-b)$} is the flat surface associated to the $1$-form $\omega$ of genus zero with a unique zero of order $a$ and a unique pole of order $b$. In the case $a=0$, the flat cone has no conical singularity and a unique pole of order two. This surface is the flat plane  where the pole of order two corresponds to the point at infinity.\newline
The neighborhood of a pole of order $b>1$ with trivial residue is the complement of a compact neighborhood of the conical singularity of the flat cone  of type $(a,-b)$.\newline
In order to get neighborhoods of poles of order $b>1$ with non trivial residue, we proceed in the following way. We take a flat cone of  type $(a,-b)$ and remove an $\epsilon$-neighborhood of a semi-infinite line starting from the conical singularity and a neighborhood of the conical singularity. Then, we identify the resulting boundaries of the neighborhood by an isometry. Rotating and rescaling gives a pole of order $b$ with the adequate residue. This construction is explained in details in~\cite{Bo}.\newline
Basically, a pole of order $b>1$ is the point at infinity for a cyclical gluing of $b-1$ flat planes along slits.\newline

\subsection{Saddle connections, Core and Period coordinates}

Every geometric notion that makes sense in the complex plane also makes sense in a translation surface away from the singularities. In particular, straight lines are well-defined as locally geodesic arcs. The circle of directions is globally defined in a translation surface. Therefore, any direction defines a directional foliation conjugated to a foliation of the complex plane in lines of the same slope.

\begin{defn} A {\em saddle connection} is a geodesic segment joining two conical singularities of the translation surface and such that all interior points are not conical singularities.
\end{defn}

Every saddle connection represents a relative homology class of $X$ punctured at the poles relatively to the zeroes. The length and the direction of a saddle connection in flat maps are respectively the modulus and the argument of the period of the meromorphic $1$-form on its homology class in $H_{1}(X \setminus \Delta,\Lambda)$.\newline

In this paper, the fact that we consider $1$-forms with a unique zero on the Riemann sphere makes everything is easier. Any saddle connection is a simple closed loop and its period is just the sum of the residues of the poles it encompasses. It is geometrically clear that every nontrivial homotopy class is represented by a unique broken geodesic formed by saddle connections and that two saddle connections representing the same homology class are the same. This implies in particular that the number of saddle connections of a translation surface belonging to these strata is at most~$\frac{p(p-1)}{2}$, where $p$ is the number of poles.\newline

Most of the geometry of a translation surface with poles is encompassed in a subsurface of finite area that is the convex hull of the conical singularities.

\begin{defn}\label{def:core}
A subset $E$ of a translation surface $(X,\omega)$ is \textit{convex} if and only if every element of any geodesic segment between two points of $E$ belongs to~$E$.\newline
The {\em convex hull} of a subset $F$ of a translation surface $(X,\omega)$ is the smallest closed convex subset of $X$ containing~$F$.\newline
The {\em core} of $(X,\omega)$ is the convex hull $core(X)$ of the conical singularities $\Lambda$ of~$\omega$. We denote by
$\mathcal{I}\mathcal{C}(X)$ the interior of $core(X)$ in $X$ and by $\partial\mathcal{C}(X) = core(X)\ \backslash\ \mathcal{I}\mathcal{C}(X)$ its boundary.
\end{defn}

The core separates the poles from each other. The following lemma shows that the complement of the core has as many connected components as there are poles. We refer to these connected components as \textit{domains of poles}. The following result is given in Proposition~4.4 and Lemma~4.5 of \cite{Ta}.

\begin{prop}\label{prop:decomppoles}
For a translation surface with $p$ poles $(X,\omega)$, the boundary of the core $\partial\mathcal{C}(X)$ is a finite union of saddle connections. Moreover, $X \setminus core(X)$ has $p$ connected components. Each of them is a topological disk that contains a unique pole.
\end{prop}

The core of a translation surface can always be triangulated by saddle connections (see \cite{Ta1} for a rigorous proof of this intuitive statement). This implies in particular that there is no deformation of translation surfaces that would not also deform the periods of the saddle connections. Moreover, it is shown in \cite{BCGGM3} that these periods give a local coordinate chart of the strata of meromorphic differentials.  Since these periods are linear combinations of the residues at the poles, the vector of the residues $(\lambda_{1},\dots,\lambda_{p-1})$ is a local coordinate for the stratum. Note that the residue theorem allows to recover the value of the last residue~$\lambda_{p}$ from the others.\newline

For strata $\mathcal{H}=\mathcal{H}(a,-b_{1},\dots,-b_{p})$ of meromorphic $1$-forms in genus zero, the isoresidual fibration $\res\colon\mathcal{H} \rightarrow \mathcal{R}_{p}$ is a finite map.   Note that this map is in general not surjective. For example, we prove in Proposition~\ref{prop:trivial} that the fiber over the uniformally zero configuration of residues is in fact empty. The image of this map is studied in details in~\cite{GT}.\newline

\subsection{Contraction flow}
\label{sec:contraction}

The group $GL^{+}(2,\mathbb{R})$ acts on each stratum of meromorphic $1$-forms by composition with coordinate functions, see \cite{Zo}. In this paper, we will only use a small part of this action.

\begin{defn} 
The {\em contraction flow} is the action of the semigroup of matrices $C^{t}=\begin{pmatrix} 1 & 0 \\ 0 & e^{-t} \end{pmatrix}$ preserving the real direction and contracting the imaginary direction.
\end{defn}

The action of the contraction flow on a translation surfaces with a unique saddle connection in the contracted direction is divergent. In the case of genus zero,  the saddle connection shrinks and the limiting surface is obtained by gluing two translation surfaces at a point. On the opposite, if there is no saddle connection in the contracted direction, the flow converges to a translation surface of the stratum where every period belongs to the preserved direction (see Subsection~5.4 of \cite{Ta} for details).\newline

Since it acts on the periods, the action of the contraction flow clearly commutes with the residue map.
Note moreover that since the resonance equations have real coefficients, the action preserves each resonance hyperplane of the residual space.\newline

\section{From translation surfaces to decorated trees}
\label{sec:transarbrees}

We begin by discussing in Section~\ref{sec:resonance} the notion of resonance arrangement introduced in the introduction. Then we describe $1$-forms with real periods in Section~\ref{sec:fibrereel}. We continue by associating a graph to these $1$-forms in Section~\ref{sec:decorated}.
Finally, we study the degenerations of $1$-forms above the resonance arrangement and associate a degenerated version of the trees to such element in Section~\ref{sec:degentree}.

\subsection{Resonance arrangement}\label{sec:resonance}

For a stratum $\mathcal{H}=\mathcal{H}(a,-b_{1},\dots,-b_{p})$ of meromorphic $1$-forms on the Riemann sphere, recall from the introduction that the {\em residual space} $\mathcal{R}_{p}$ is the complex vector space $\mathcal{R}_{p}$ of dimension $p-1$ formed by the $p$-vectors $(\lambda_{1},\dots,\lambda_{p})$ such that $\sum_{j=1}^{p} \lambda_{j} = 0$. Note that this vector space only depends on the number of poles $p$. Note moreover that this definition slightly differ from the one of \cite{GT}.\newline 

We now set some important notation.
\begin{defn}\label{def:resonance}
 Let $\mu=(a,-b_{1},\dots,-b_{p})$ be a partition of $-2$ and $I$ a subset of $\{1,\dots,p\}$, we denote by $d_{I}=-1+\sum_{j \in I} b_{j}$ the \textit{resonance degree} of $I$ and by $c_{I}$ the cardinal of $I$. 
\end{defn}
Note that we have $d_{I}+d_{I^{\complement}}=a$ and $c_{I}+c_{I^{\complement}}=p$.  

Recall from Definition~\ref{def:resonanceintro} that for any non trivial subset $I$ of $\{1,\dots,p\}$ we define a hyperplane of the residual space giving by the equation $\sum_{i\in I}\lambda_{i}$. Note that the resonance hyperplane defined by $I$ or its complement $I^{\complement}$ is the same. The union of these hyperplanes gives the  {\em resonance arrangement}, also known in the literature as the \textit{restricted all-subset arrangement}, see \cite{AM}. More precisely, for each non empty strict subset $I$ of $\lbrace 1,\dots,p \rbrace$ we define the hyperplane
$$A_{I} = \left\{ (\lambda_{1},\dots,\lambda_{p})\in\mathcal{R}_{p}: \sum_{i\in I} \lambda_{i} = 0 \right\}\,.$$
Note that $A_{I} = A_{I^{\complement}}$, where $I  \sqcup I^{\complement} =\lbrace 1,\dots,p \rbrace $.
The union $\mathcal{A}_{p}$ of all resonance hyperplanes~$A_{I}$ is a hyperplane arrangement in $\mathcal{R}_{p}$. It will be proved in Theorem~\ref{thm:dependarrangement} that the number of elements of an isoresidual fiber $\mathcal{F}$ only depends on the family of resonance hyperplanes $\mathcal{F}$ belongs to.\newline
In particular, the complement of the hyperplane arrangement is connected. In accordance with Arnold's Italian principle (see \cite[p. 64]{A}), we will obtain that every fiber of $\mathcal{R}_{p} \setminus \mathcal{A}_{p}$ has the same number of elements.\newline

If $p \leq 3$, the resonance arrangements are well understood. The case $p=2$ is 
\[ \mathcal{A}_{2} = \left\{ \lambda_{1} = 0 \right\} \subset \mathbb{C} \,.\]
So the arrangement complement $\mathcal{R}_{2} \setminus \mathcal{A}_{2}$ is isomorphic to $\mathbb{C}^{\ast}$.  
In the case $p=3$, $\mathcal{A}_{3}$ is formed by three complex hyperplanes with trivial mutual intersection in a complex vector space of dimension two as shown in Figure~\ref{fig:arrangement}.

Now we show that for $p\geq 4$, the hyperplane arrangement is of special kind. A {\em simplicial arrangement} is the complexification of an arrangement of real hyperplanes where some chambers are simplicial cones (they are cut out by the minimal number $p-1$ of hyperplanes).

\begin{lem}\label{lem:passimplice}
If $p \geq 4$, the resonance arrangement $\mathcal{A}_{p}$ is not a simplicial arrangement. 
\end{lem}
 
\begin{proof}
In the case $p=4$, the chamber defined by the seven inequalities $$z_{1},z_{2},z_{1}+z_{2},z_{1}+z_{3},z_{1}+z_{4},-z_{3},-z_{4} \geq 0$$ is bounded by the four resonance hyperplanes $A_{3}$, $A_{4}$, $A_{\{2,3\}}$ and $A_{\{2 , 4\}}$. 
This example is easily generalized to prove the case $p\geq5$.
\end{proof}
\smallskip
\par
Since the hyperplanes of the resonance arrangement are defined by linear forms with real coefficients, it induces an arrangement $\mathbb{R}\mathcal{A}_{p}$ of real hyperplanes in the subspace $\mathbb{R}\mathcal{R}_{p} = \left\{ (\lambda_{1},\dots, \lambda_{p}) \in \mathbb{R}^{p}: \sum \lambda_{i} = 0\right\} \subset \mathcal{R}_{p}$ formed by real configuration of residues.   This real arrangement has exactly the same incidence structure as the initial complex arrangement. In this case, any partial sum of residues can be negative, positive or zero. These real linear forms cuts out $\mathbb{R}\mathcal{R}_{p}$ into chambers and walls. We encompass this information into a sign function $\psi_{\mathcal{F}}$ associated to a real fiber $\mathcal{F}$ that assigns to any subset $I$ of $\{1,\dots,p\}$ the sign $\psi_{\mathcal{F}}(I)$ of the partial sum of the corresponding residues. We will equally see this function as a function of the chambers of the real arrangement. We get a morphism from the set $\mathcal{P}(1,\dots,p)$ of non trivial (i.e.  non empty and not the whole set) subsets  of $\{1,\dots,p\}$  to the hyperfield of signs $\{-,0,+\}$, see \cite{Vi}. Note that two real fibers whose sign functions vanish on the same subset belong to the same family of resonance hyperplanes.\newline

 Finally, we present these notions in an easy example.
  \begin{ex}\label{ex:arrangement}
 In this example we discuss the case of the stratum $\mathcal{H}(4;-2,-2,-2)$. 
 
 In Figure~\ref{fig:arrangement}, we show the resonance arrangement $\mathcal{A}_{3}$ in the residue space $\mathcal{R}_{3}$. This has~$3$ resonance hyperplanes $A_{1}$, $A_{2}$ and $A_{1,2}$. 
   \begin{figure}[ht]
\begin{tikzpicture}[scale=1]
   \fill[fill=black!10] (0,0) circle (2cm);
   \draw (-2,0) -- (2,0) coordinate[pos=.1] (a);
    \draw (0,-2) -- (0,2) coordinate[pos=.1] (b);
    \draw (135:2) -- (-45:2) coordinate[pos=.9] (c);
    \node[above] at (a) {$A_{2}$};
    \node[left] at (b) {$A_{1}$};
    \node[below] at (c) {$A_{1,2}$};
    
     \fill (1,-1/2)  circle (1pt);

\end{tikzpicture}
 \caption{The resonance arrangement in the plane $(\lambda_{1},\lambda_{2})$.} \label{fig:arrangement}
\end{figure}
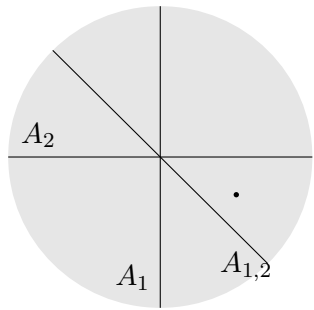

The isoresidual fiber in the complement of $\mathcal{A}_{3}$ has~$4$ elements. Let us pick a point in the complement, for example the point $(2,-1,-1)$ as pictured in Figure~\ref{fig:arrangement}. Note that the sign function $\psi$ at this point is
\[ \psi(1) = +,\ \psi(2) = -,\ \psi(3) = -,\ \psi(1,2) = +,\ \psi(1,3) = +,\ \psi(2,3) = -\,.\] Two of the four differentials are pictured in Figure~\ref{fig:arrangement1} as we will explain in the next paragraph. The two other differentials are obtained by permuting the role of $p_{2}$ and $p_{3}$. 
   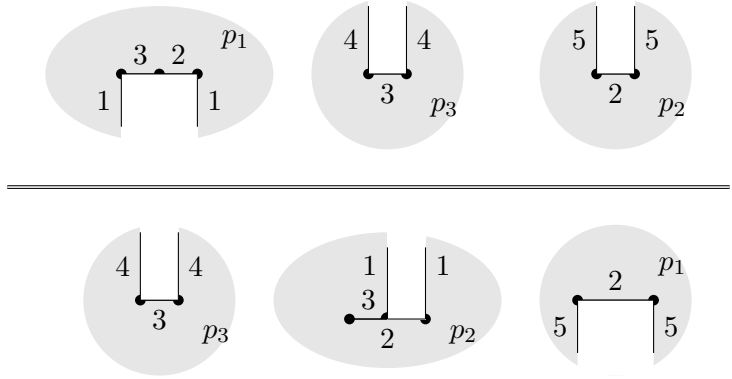
\begin{figure}[ht]
\begin{tikzpicture}[scale=1]
      \fill[fill=black!10] (0,0) ellipse (1.5cm and .9cm);
   \draw (-.5,0) coordinate (a) -- node [above] {$3$} 
(0,0) coordinate (b)  -- node [above] {$2$} 
(.5,0) coordinate (c);
 \fill (a)  circle (2pt);
\fill[] (b) circle (2pt);
\fill[] (c) circle (2pt);

    \fill[white] (a) -- (c) -- ++(0,-1) --++(-1,0) -- cycle;
    \draw (a) -- (c);
 \draw  (a)  -- ++(0,-.7) coordinate[pos=.5] (e);
\draw  (c)  -- ++(0,-.7) coordinate[pos=.5] (f);
 \node[left] at (e) {$1$};
 \node[right] at (f) {$1$};
 \node at (1,.45) {$p_{1}$};
    
\begin{scope}[xshift=3cm]
      \fill[fill=black!10] (0,0) circle (1cm);
      \draw (-.25,0) coordinate (a) -- (.25,0) coordinate[pos=.5] (c) coordinate (b);
  \fill[] (a) circle (2pt);
\fill[] (b) circle (2pt);
    \fill[white] (a) -- (b) -- ++(0,1) --++(-.5,0) -- cycle;
 \draw  (b) --  (a);
 \draw (a) -- ++(0,.9) coordinate[pos=.5] (d);
 \draw (b) -- ++(0,.9)coordinate[pos=.5] (e);
\node[below] at (c) {$3$};
\node[left] at (d) {$4$};
\node[right] at (e) {$4$};
 \node at (.75,-.45) {$p_{3}$};
    \end{scope} 
    
    \begin{scope}[xshift=6cm]
      \fill[fill=black!10] (0,0) circle (1cm);
      \draw (-.25,0) coordinate (a) -- (.25,0) coordinate[pos=.5] (c) coordinate (b);
  \fill[] (a) circle (2pt);
\fill[] (b) circle (2pt);
    \fill[white] (a) -- (b) -- ++(0,1) --++(-.5,0) -- cycle;
 \draw  (b) --  (a);
 \draw (a) -- ++(0,.9) coordinate[pos=.5] (d);
 \draw (b) -- ++(0,.9)coordinate[pos=.5] (e);
\node[below] at (c) {$2$};
\node[left] at (d) {$5$};
\node[right] at (e) {$5$};
 \node at (.75,-.45) {$p_{2}$};
    \end{scope}    

\draw[double] (-2,-1.5) -- (7.5,-1.5);

  \begin{scope}[xshift=0cm,yshift=-3cm]

      \fill[fill=black!10] (0,0) circle (1cm);
      \draw (-.25,0) coordinate (a) -- (.25,0) coordinate[pos=.5] (c) coordinate (b);
  \fill[] (a) circle (2pt);
\fill[] (b) circle (2pt);
    \fill[white] (a) -- (b) -- ++(0,1) --++(-.5,0) -- cycle;
 \draw  (b) --  (a);
 \draw (a) -- ++(0,.9) coordinate[pos=.5] (d);
 \draw (b) -- ++(0,.9)coordinate[pos=.5] (e);
\node[below] at (c) {$3$};
\node[left] at (d) {$4$};
\node[right] at (e) {$4$};
 \node at (.75,-.45) {$p_{3}$};

\begin{scope}[xshift=3cm]
       \fill[fill=black!10] (0,0) ellipse (1.5cm and .9cm);
   \draw (-.5,-.25) coordinate (a) -- node [above] {$3$} 
(0,-.25) coordinate (b);
\draw  (a)  -- node [below] {$2$} (.5,-.25) coordinate (c);
 \fill (a)  circle (2pt);
\fill (b) -- (0,-.16) arc [start angle=90, end angle=180, radius=.9mm] -- ++(.9,0);
\fill[] (c) circle (2pt);

    \fill[white] (b) -- (c) -- ++(0,1.3) --++(-.5,0) -- cycle;
    \draw (b) -- (c);
 \draw  (b)  -- ++(0,.95) coordinate[pos=.75] (e);
\draw  (c)  -- ++(0,.95) coordinate[pos=.75] (f);
 \node[left] at (e) {$1$};
 \node[right] at (f) {$1$};
 \node at (1,-.45) {$p_{2}$};
    \end{scope} 
    
    \begin{scope}[xshift=6cm]
      \fill[fill=black!10] (0,0) circle (1cm);
      \draw (-.5,0) coordinate (a) -- (.5,0) coordinate[pos=.5] (c) coordinate (b);
  \fill[] (a) circle (2pt);
\fill[] (b) circle (2pt);
    \fill[white] (a) -- (b) -- ++(0,-1) --++(-1,0) -- cycle;
 \draw  (b) --  (a);
 \draw (a) -- ++(0,-.7) coordinate[pos=.5] (d);
 \draw (b) -- ++(0,-.7)coordinate[pos=.5] (e);
\node[above] at (c) {$2$};
\node[left] at (d) {$5$};
\node[right] at (e) {$5$};
 \node at (.75,.45) {$p_{1}$};
    \end{scope}    
    \end{scope}
\end{tikzpicture}
 \caption{Two of the four elements of the isoresidual fiber above the point $(2,-1,-1)$.} \label{fig:arrangement1}
\end{figure}

Let us explain Figure~\ref{fig:arrangement1}. Consider the three pictures above the line representing one differential in the fiber of $(2,-1,-1)$. Each of the three discs represents the local model of a pole of order $-2$, as explained in Section~\ref{sec:locmod}. Then we glue by horizontal translations the half-lines with the same labels $1$, $4$ and $5$. Now the two pairs of segments $2$ and $3$ are glued together by translations. The zero of order $4$ is pictured by the black discs.

Let us now explain why the pole $p_{1}$ has residue $2$. Take the loop bounding the disc represented the pole $p_{1}$ in Figure~\ref{fig:arrangement1} oriented in the indirect direction. By residue theorem the integral of the $1$-form is equal to the residue at $p_{1}$ (after rescaling by $2i\pi$ as explained in Section~\ref{sec:locmod}). 
  \end{ex}

\subsection{Real fibers}
\label{sec:fibrereel}

The translation surfaces in the fibers over real residues are especially easy to understand. In this subsection, we give their specific properties. They are a special case of  real-normalized meromorphic differentials of Grushevsky and Krichever, see \cite{GK1,GK2}. Besides, the square of any such $1$-form also belongs to the well-studied class of Strebel differentials, see \cite{St}. We also explain in what sense the isoresidual fibration can be understood by considering only real fibers.

Recall from Definition~\ref{def:core} that the core of a translation surface is the convex hull of the conical singularities.
\begin{lem}\label{lem:decompreel}
A meromorphic $1$-form in a stratum $\mathcal{H}(a,-b_1,\dots,-b_p)$ with real residues defines a translation surface with a degenerate core. The latter is the union of $p$ domains of poles, $p-1$ horizontal saddle connections and one conical singularity.
\end{lem}

\begin{proof}
The period of any saddle connection is real so all of them are horizontal. The core (see Definition~\ref{def:core}) thus does not contain any nondegenerate triangle and is thus itself degenerate. The core is then an embedded graph with one vertex corresponding to the unique zero and several edges corresponding to the saddle connections. According to Proposition \ref{prop:decomppoles}, the core cuts out the surface into $p$ domains of poles. The computation of Euler characteristic proves that there are $p-1$ saddle connections.
\end{proof}

The following proposition reduces the understanding of general fibers to fibers over a real configuration of residues which belongs to the same set of hyperplanes. 

\begin{prop}\label{prop:realgen}
Let $\mathcal{F}$ be a fiber in a stratum $\mathcal{H}(a;-b_{1},\dots,-b_{p})$ of meromorphic $1$-forms on the Riemann sphere with residues $\lambda=(\lambda_{1},\dots,\lambda_{p})$. Let $B$ be the set of resonance hyperplanes containing $\lambda$. There is a fiber $\mathcal{F'}$ over a configuration of real residues that is contained in $B$ and such that $\mathcal{F}$ and $\mathcal{F'}$ are in bijection.
\end{prop}

\begin{proof}
We consider an isoresidual fiber $\mathcal{F}$. Up to conjugacy by the action of $GL^{+}(2,\mathbb{R})$, we can assume no partial sum of residues is purely imaginary. Then, we use the contraction flow preserving the real direction and contracting the imaginary direction. In particular, for any translation surface of $\mathcal{F}$, no saddle connection belongs to the contracted direction (the period of every saddle connection of a translation surface of genus zero with only one conical singularity is a partial sum of residues). Consequently, the contraction flow converges for every element of the fiber~$\mathcal{F}$.\newline
Since the contraction flow commutes with the residual map, the limit points of the contraction flow belong to the same fiber $\mathcal{F'}$ over a configuration of real residues. Moreover, this configuration satisfies the same resonance equations as that of $\mathcal{F}$. Indeed, the set of resonance hyperplanes containing the residues of $\mathcal{F'}$ clearly contains $B$. Besides, since no partial sum of residues is purely imaginary, no additional resonance equation is satisfied by the residues in fiber $\mathcal{F'}$.\newline
The contraction flow induces a map $f\colon\mathcal{F}\to\mathcal{F'}$. To conclude the proof of Proposition~\ref{prop:realgen}, it suffices to prove that this map is an bijection.\newline
For every element $x$ of $\mathcal{F'}$, residues $(\lambda_{1},\dots,\lambda_{p-1})$ define a local biholomorphism between a neighborhood $\mathcal{V}(x)$ of $x$ in the stratum and some open subset of $\mathbb{C}^{p-1}$. Since $\mathcal{F}$ is a finite set, we define for each $x \in \mathcal{F}$ a neighborhood $\mathcal{V}(x)$ such that $\mathcal{V}(x)$ and $\mathcal{V}(y)$ are disjoint if $x \neq y$.\newline
It is rather clear that for each $x \in X$ and some large enough $t$, neighborhood $\mathcal{V}(x)$ contains exactly one element of fiber $C^{t}(\mathcal{F})$ (conjugated to $\mathcal{F}$ by the contraction flow). Since the contraction flow commutes with the residue map, this implies directly that $f$ is surjective.\newline
If $f$ fails to be injective, then there are two elements $z,z' \in \mathcal{F}$ such that their limit under the contraction flow is $x \in \mathcal{F'}$. Then, for arbitrarily small neighborhood $\mathcal{V}$ of $x$ and some arbitrarily large $t$, $\mathcal{V}$ contains two distinct elements of $C^{t}(\mathcal{F})$. This contradicts local injectivity of period coordinates. Thus $f$ is a bijection between $\mathcal{F}$ and $\mathcal{F'}$.\newline
\end{proof}

\subsection{Decorated trees}
\label{sec:decorated}

We begin this section by introducing the notion of \textit{decorated tree} that will classify combinatorially the real $1$-forms of the strata $\mathcal{H}(a,-b_{1},\dots,-b_{p})$. Some decorated trees are pictured in Figure~\ref{fig:arrangement2}.

\begin{defn}\label{def:abstractdecoratedtree}
A \textit{decorated tree} is an embedded directed tree in the topological sphere such that every vertex is labeled and to every vertex is attached a nonnegative even number of unoriented half-edges. Moreover, there is a nonnegative even number of  half-edges between two adjacent edges with the same direction (at the vertex) and an odd number of half-edges between two edges of opposite directions.
\end{defn}



For any $1$-form $\omega$ of a stratum $\mathcal{H}(a,-b_{1},\dots,-b_{p})$ with non zero real residues, the horizontal trajectories starting from the unique conical singularity form an embedded graph in the sphere. Since a zero of order $a$ corresponds to a conical singularity of angle $(2a+2)\pi$, there are $2a+2$ such horizontal trajectories. These horizontal trajectories either converge to a pole or form a saddle connection. Since we consider trajectories in positive and negative directions, a saddle connection counts as a pair of horizontal trajectories. According to Lemma~\ref{lem:decompreel}, there are $p-1$ saddle connections in~$\omega$. Therefore, the \textit{horizontal graph} is formed by $p-1$ closed saddle connections and $2a-2p+4$ trajectories going to a pole.\newline
Since the local model a pole of order $b_{j}>1$ is the cyclic gluing of $b_{j}-1$ planes, there are exactly $2b_{j}-2$ horizontal trajectories going to the pole of order $b_{j}$. Note that this is consistent with the identity $\sum_{j=1}^{p} b_{j} = a+2$.\newline

It is more convenient to consider the following graph that is closely related to the dual graph of the horizontal graph. We will call it the \textit{decorated tree} of the $1$-form. 

\begin{defn}\label{def:difftodecoratedtree}
Given a meromorphic $1$-form $\omega$ with real residues in $\mathcal{H}(a,-b_{1},\dots,-b_{p})$, its \textit{decorated tree}  $\mathfrak{t}(\omega)$ is the decorated tree such that:
\begin{enumerate}
	\item its vertices correspond to the $p$ poles of $\omega$ and have the corresponding labels;
	\item its edges correspond to saddle connections between two domains of poles and the orientation is such that it goes from the lower domain to the upper domain;
	\item its half-edges correspond to horizontal trajectories from the zero to this pole;
	\item its embedding is such that the order on the set of horizontal trajectories at the zero given by going positively around the zero is identical to the order on the edges and half-edges going positively around the tree.
\end{enumerate}
\end{defn}

Let us justify that a decorated  tree associated to a differential with real residues is a decorated tree as defined in Definition~\ref{def:abstractdecoratedtree}. First note that there are $2b_{j}-2$ half-edges attached to the vertex~$j$. Moreover, the only point of Definition~\ref{def:abstractdecoratedtree} which is not clearly satisfied by $\mathfrak{t}(\omega)$ is the fact that there is a  even number of half-edges between two edges with the same direction and an odd number of half-edges between two edges of opposite directions. This fact follows from the fact that two consecutive edges with opposite directions  correspond to saddle connections which are locally in the same horizontal direction. \newline
We give an example of decorated tree associated to some meromorphic $1$-forms.
\begin{ex}\label{ex:arrangement2}
 We continue to study the stratum $\mathcal{H}(4;-2,-2,-2)$ as in  Example~\ref{ex:arrangement}. A point of the residual space which we consider lies in the lower right chamber between the~$A_{2}$ and $A_{1,2}$.   Consider the two differentials pictured in the right of Figure~\ref{fig:arrangement1}. The decorated trees associated to these differentials are pictured in Figure~\ref{fig:arrangement2}, the left graph corresponding to the top differential as explained now.
 \begin{figure}[ht]
\begin{tikzpicture}[scale=1,decoration={
    markings,
    mark=at position 0.5 with {\arrow[very thick]{>}}}]
   
\node[circle,draw] (a) at  (-2,0) {$p_{3}$};
\node[circle,draw] (b) at  (0,0) {$p_{1}$};
\node[circle,draw] (c) at  (2,0) {$p_{2}$};
\draw [postaction={decorate}]  (a) --  (b);
\draw[postaction={decorate}] (c) -- (b);

\draw (a) -- ++(120:.6);
\draw (a) -- ++(240:.6);
\draw (c) -- ++(60:.6);
\draw(c) -- ++(-60:.6);
\draw(b) -- ++(-120:.6);
\draw (b) -- ++(-60:.6);


\begin{scope}[xshift=7cm]
\node[circle,draw] (a) at  (-2,0) {$p_{1}$};
\node[circle,draw] (b) at  (0,0) {$p_{2}$};
\node[circle,draw] (c) at  (2,0) {$p_{3}$};
\draw [postaction={decorate}]  (b) --  (a);
\draw[postaction={decorate}] (c) -- (b);

 \draw (a) -- ++(120:.6);
\draw (a) -- ++(240:.6);
\draw (c) -- ++(60:.6);
\draw (c) -- ++(-60:.6);
\draw (b) -- ++(90:.6);
\draw (b) -- ++(-90:.6);

\end{scope}

 \end{tikzpicture}
 \caption{The decorated trees associated to the differentials of Figure~\ref{fig:arrangement1}} \label{fig:arrangement2}
\end{figure}
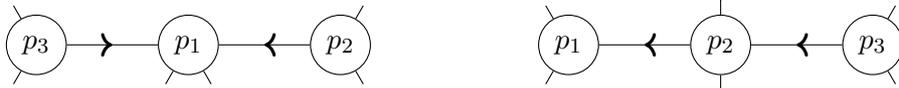

The pole $p_{1}$ is adjacent to each $p_{2}$ and $p_{3}$, so there is an edge connecting the vertex $p_{1}$ to the vertices $p_{2}$ and $p_{3}$. Since we go up to pass from $p_{i}$ with $i=2,3$ to the pole $p_{1}$ the orientations of these edges are from $p_{i}$ to $p_{1}$. Since all the poles are of order $-2$, each vertex has two half-edges. It remains to justify that both half-edges at the vertex $p_{1}$ are as pictured. There is no horizontal trajectory going from the intersection of the saddle connections corresponding to the segments $2$ and $3$ going to the pole $p_{1}$. One the other hand, there are two such trajectories, going left of the saddle $3$ and right of the saddle $2$. So Points~(3) and~(4) of Definition~\ref{def:difftodecoratedtree} implies that the half-edges are as pictured.
\end{ex}

The geometry of a $1$-form with real residues can be read on the associated decorated tree in an easy way. The periods are the (real positive) lengths of the saddle connections (one for each edge) and pair of consecutive edges or half-edges corresponds to an angle of~$\pi$. In a given stratum $\mathcal{H}(a,-b_{1},\dots,-b_{p})$, the $1$-forms with real residues that share the same decorated tree form a convex polytope of real dimension $p-1$ in the period coordinates of the stratum. This locus is globally defined by the inequalities saying that the lengths of saddle connections are strictly positive.\newline

It is clear that in an isoresidual fiber, two $1$-forms with the same decorated tree are equal (the periods allow to reconstruct the translation surface from the tree in unique way). Therefore, counting the elements of an isoresidual fiber amounts to enumerate decorated trees that are consistent with the configuration of residues. In particular, every fiber has finitely many elements.\newline

We first define abstractly the set of decorated trees compatible with a distribution of degrees and a sign function.

\begin{defn}\label{def:merodecoratedtree}
For any partition $(a,b_{1},\dots,b_{p})$ of $-2$ and any sign function $\psi$, a {\em  decorated tree compatible with $\psi$}  is a decorated tree $\mathfrak{t}$  such that:
\begin{enumerate}
	\item it has $p$ vertices labelled by numbers from $1$ to $p$;
	\item there are $2b_{j}-2$ half-edges on each vertex~$j$;
	\item each edge $e$ cuts $\mathfrak{t}$ into two subtrees with vertices corresponding to $I$ and $I^{\complement}$ in $\left\{1,\dots,p\right\}$ such that $e$ is oriented from the subtree $I$ to the subtree $I^{\complement}$, and $\psi(I)<0$ and $\psi(I^{\complement})>0$.
\end{enumerate}
The set of decorated trees compatible with $\psi$ is denoted by $\mathcal{T}(b_{1},\dots,b_{p},\psi)$.
\end{defn}

The importance of this notion is given by the following result showing that the set $\mathcal{T}(b_{1},\dots,b_{p},\psi)$ is in bijection with the isoresidual fiber $\mathcal{F}$ in the corresponding strata over the configuration of real residues such that $\psi = \psi_{\mathcal{F}}$. We begin with the non resonant case.

\begin{lem}\label{lem:arbresdiff}
Given a stratum $\mathcal{H}(a,-b_{1},\dots,-b_{p})$ of meromorphic $1$-forms on the Riemann sphere,  there is a bijection between an isoresidual fiber $\mathcal{F}$ over a configuration of real residues in the complement of the resonance arrangement and the set of decorated trees of $\mathcal{T}(b_{1},\dots,b_{p},\psi_{\mathcal{F}})$. 
\end{lem}

\begin{proof}
Given real residues $\lambda=(\lambda_{1},\dots,\lambda_{p})$, the meromorphic $1$-forms in the isoresidual fiber of $\lambda$ are characterized by their decorated trees. Indeed, given a decorated tree $\mathfrak{t}$, we can construct a unique $1$-form in the following way. Take a leaf of $\mathfrak{t}$ with label~$i$. We associate to this leaf a polar domain of type $b_{i}$ with residue~$\lambda_{i}$. This operation is performed at each leaf of $\mathfrak{t}$ and we delete these leaf. Let us now consider a vertex, with label $j$, which correspond to a leaf of this new graph. There are two edges $e_{1}$ and $e_{2}$ which connect $j$ to the rest of the graph. We suppose that $e_{1}$ connects $j$ to the leaf $i$ and $e_{2}$ to the rest of the graph. Suppose that the two edges are consecutive. Then there are in the same direction, say going to $j$, and we take an half-plane with boundary the segments $r_{i}$ and $r_{j}-r_{i}$. Note that  Condition~(3) of Definition~\ref{def:merodecoratedtree}  ensures that $r_{j}-r_{i} \neq 0$. Then we glue to this half-plane a lower half-plane and finally glue $b_{i}-2$ planes to form a polar domain. If the edges are in the same direction but there are an even number $2a$ of half-edges in between, we take an upper-plane with boundary the segment $r_{i}$, then $a$ planes, then  an half-plane with boundary the segment $r_{j}-r_{i}$ and then the other planes. Finally if~$e_{1}$ and~$e_{2}$ are in opposite direction, we make the same construction with one upper half-plane and one lower half-plane. We glue the boundary $r_{i}$ of this polar domain to the one associated to~$i$. We continue this procedure to obtain a well-define $1$-form in this isoresidual cover. Clearly, this invert the procedure of associating a decorated tree from a $1$-form.
\end{proof}

\subsection{Resonance locus and degeneration}
\label{sec:degentree}
In order to study the isoresidual fibration above the resonant locus, it is useful to discuss degenerations of $1$-forms on the sphere. In the following, we do not consider the origin of the residual space, since  this case is very special. \newline

Above the resonance hyperplanes some decorated trees can not be realized  by $1$-forms since an edge corresponding to a resonance could have zero length.  There are various way to associate an object in these cases.  The most naive one is to quit these edges and to associate to each subtree a differential as in Lemma~\ref{lem:arbresdiff}. This corresponds to taking the WYSIWYG closure of the family and has its advantages and its drawbacks (see \cite{CW}). We denote by $\bar{\mathcal{H}}(a,-b_{1},\dots,-b_{p})$ the WYSIWYG closure of the stratum  $\mathcal{H}(a,-b_{1},\dots,-b_{p})$.

We now explain how decorated tree with sign functions that vanish on some subsets  lead to the elements of $\mathcal{H}(a,-b_{1},\dots,-b_{p})$. Given a decorated tree $\mathfrak{t}$ and a sign function $\psi$, we consider the set $E$ of edges of $\mathfrak{t}$  that separate the vertices into two subset $I$ and $I^{\complement}$ such that $\psi(I) = 0$. Then we delete the edges of $E$ of $\mathfrak{t}$ obtaining a set of disconnected trees. It is easy to see that these trees or decorated trees. Now on each tree, we can define a sign function by restriction of~$\psi$. The $1$-forms given by these decorated trees compatible with these sign function give an element of  $\mathcal{H}(a,-b_{1},\dots,-b_{p})$. Note that there is now correspondence since this procedure forgets about the way these $1$-forms are glued together.
In particular, we proved the following useful result.

\begin{lem}\label{lem:corWYSIWYG}
 Given a decorated tree $\mathfrak{t}$ and a sign function $\psi$, the differential corresponding to these data is singular if and only if there exists $I$ a strict subset of $\{1,\dots,p\}$ such that $\psi(I) = 0$ and there exists an edge of $\mathfrak{t}$ given two subtrees with vertices corresponding respectively to $I$ and $I^{\complement}$. 
\end{lem}

Using this notion, we can extend Lemma~\ref{lem:arbresdiff} without problems to the case above the resonance arrangement.  This will be an important tool in Section~\ref{sec:comptereson}.
\begin{lem}\label{lem:arbresdiffdegen}
Given a stratum $\mathcal{H}(a,-b_{1},\dots,-b_{p})$ of meromorphic $1$-forms on the Riemann sphere,  there is a bijection between an isoresidual fiber $\mathcal{F}$ over a configuration of real residues in $\mathcal{H}(a,-b_{1},\dots,-b_{p})$  and the set of decorated trees $\mathfrak{t}$ of $\mathcal{T}(b_{1},\dots,b_{p},\psi_{\mathcal{F}})$ for which there is no edge separating $\mathfrak{t}$ into two subtrees with $\psi_{\mathcal{F}}(I)=0$ for $I$ corresponding to the vertices of one subtree. 
\end{lem}

We give more examples of degeneration. First, if the resonance condition consists in the vanishing of the residue of a simple pole, then the degeneration is a differential with a simple pole less and a zero whose order has decreased by one.   Now we continue the case of the stratum $\mathcal{H}(4,-2,-2,-2)$ began in Example~\ref{ex:arrangement3}.

\begin{ex}\label{ex:arrangement3}
 We continue Examples~\ref{ex:arrangement} and~\ref{ex:arrangement2}. In Figure~\ref{fig:arrangement3}, we give the elements in $\bar{\mathcal{H}}(4,-2,-2,-2)$ lying above a point of the hyperplane~$A_{2}$. In this case, two elements  are $1$-forms in $\mathcal{H}(4,-2,-2,-2)$ (the second is changing the role of $p_{2}$ and $p_{3}$) and the other element is singular. We can easily check  that $3$ families of $1$-forms converge to the singular element and only $1$ to the $1$-form. This phenomena will be explained in Section~\ref{sec:comptereson}.
   \begin{figure}[ht]
\begin{tikzpicture}[scale=1,decoration={
    markings,
    mark=at position 0.5 with {\arrow[very thick]{>}}}]
   \fill[fill=black!10] (0,0) circle (2cm);
   \draw (-2,0) -- (2,0) coordinate[pos=.1] (a);
    \draw (0,-2) -- (0,2) coordinate[pos=.1] (b);
    \draw (135:2) -- (-45:2) coordinate[pos=.9] (c);
    \node[above] at (a) {$A_{2}$};
    \node[left] at (b) {$A_{1}$};
    \node[below] at (c) {$A_{1,2}$};
    
     \fill (1,-.5)  circle (1pt);
    \draw[postaction={decorate}] (1,-.45) -- (1,-.05) coordinate[pos=.5] (e);
    \node[left] at (e) {$ \alpha$};   
         \fill (1,0)  circle (1pt);
    \draw[->] (3.5,0) .. controls ++(-140:1) and ++(-40:1) .. (1.05,-0.05) ;

    \begin{scope}[xshift=5cm,yshift=1.3cm]

      \fill[fill=black!10] (0,0) circle (1cm);
   \draw (-.25,0) coordinate (a) -- node [above] {$3$} 
(.25,0) coordinate (b) ;
 \fill (a)  circle (2pt);
\fill[] (b) circle (2pt);

    \fill[white] (a) -- (b) -- ++(0,-1) --++(-.5,0) -- cycle;
    \draw (a) -- (b);
 \draw  (a)  -- ++(0,-.9) coordinate[pos=.5] (e);
\draw  (b)  -- ++(0,-.9) coordinate[pos=.5] (f);
 \node[left] at (e) {$1$};
 \node[right] at (f) {$1$};
 \node at (1,.45) {$p_{1}$};
    
\begin{scope}[xshift=3cm]
      \fill[fill=black!10] (0,0) circle (1cm);
      \draw (-.25,0) coordinate (a) -- (.25,0) coordinate[pos=.5] (c) coordinate (b);
  \fill[] (a) circle (2pt);
\fill[] (b) circle (2pt);
    \fill[white] (a) -- (b) -- ++(0,1) --++(-.5,0) -- cycle;
 \draw  (b) --  (a);
 \draw (a) -- ++(0,.9) coordinate[pos=.5] (d);
 \draw (b) -- ++(0,.9)coordinate[pos=.5] (e);
\node[below] at (c) {$3$};
\node[left] at (d) {$4$};
\node[right] at (e) {$4$};
 \node at (.75,-.45) {$p_{3}$};
    \end{scope} 
    
    \begin{scope}[xshift=6cm]
      \fill[fill=black!10] (0,0) circle (1cm);
    
  \filldraw[fill=white] (0,0) circle (2pt);
 \node at (.75,-.45) {$p_{2}$};
    \end{scope}    
\end{scope}

\draw[double] (4,0) -- (12,0);


  \begin{scope}[xshift=5cm,yshift=-1.2cm]

      \fill[fill=black!10] (0,0) circle (1cm);
      \draw (-.25,0) coordinate (a) -- (.25,0) coordinate[pos=.5] (c) coordinate (b);
  \fill[] (a) circle (2pt);
\fill[] (b) circle (2pt);
    \fill[white] (a) -- (b) -- ++(0,1) --++(-.5,0) -- cycle;
 \draw  (b) --  (a);
 \draw (a) -- ++(0,.9) coordinate[pos=.5] (d);
 \draw (b) -- ++(0,.9)coordinate[pos=.5] (e);
\node[below] at (c) {$3$};
\node[left] at (d) {$4$};
\node[right] at (e) {$4$};
 \node at (.75,-.45) {$p_{3}$};

\begin{scope}[xshift=3cm]
       \fill[fill=black!10] (0,0) circle (1cm);
   \draw (-.25,0) coordinate (a) -- node [above] {$3$} node [below] {$2$} 
(.25,0) coordinate (b) ;
 \fill (a)  circle (2pt);
\fill[] (b) circle (2pt);

 \node at (.75,-.45) {$p_{2}$};
    \end{scope} 
    
    \begin{scope}[xshift=6cm]
      \fill[fill=black!10] (0,0) circle (1cm);
      \draw (-.25,0) coordinate (a) -- (.25,0) coordinate[pos=.5] (c) coordinate (b);
  \fill[] (a) circle (2pt);
\fill[] (b) circle (2pt);
    \fill[white] (a) -- (b) -- ++(0,-1) --++(-.5,0) -- cycle;
 \draw  (b) --  (a);
 \draw (a) -- ++(0,-.9) coordinate[pos=.5] (d);
 \draw (b) -- ++(0,-.9)coordinate[pos=.5] (e);
\node[above] at (c) {$2$};
\node[left] at (d) {$5$};
\node[right] at (e) {$5$};
 \node at (.75,.45) {$p_{1}$};
    \end{scope}    
\end{scope}
\end{tikzpicture}
 \caption{The isoresidual fiber in $\bar{\mathcal{H}}(4,-2,-2,-2)$ above the resonance hyperplane $A_{2}$.} \label{fig:arrangement3}
\end{figure}
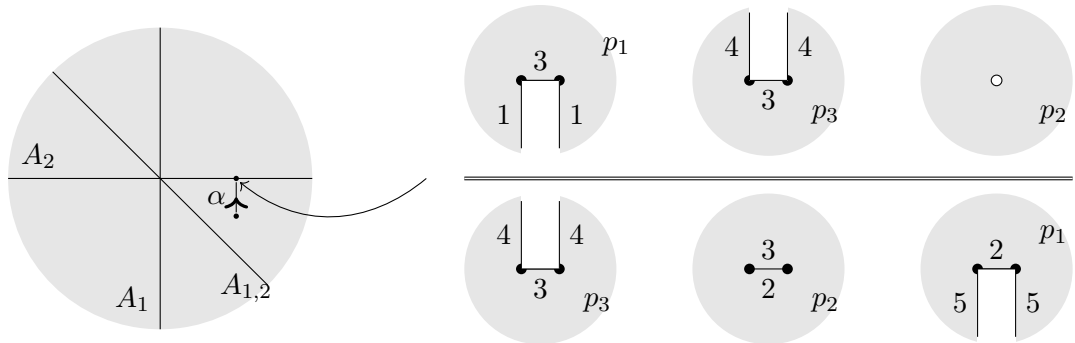
 If we consider the family of differentials above~$\alpha$, the differential pictured on the top, resp. the bottom, of Figure~\ref{fig:arrangement1} converges to the top, resp. bottom, differential pictured in Figure~\ref{fig:arrangement3}.  Finally, at this point, the sign function $\psi$ is
\[ \psi(1) = +,\ \psi(2) = 0,\ \psi(3) = -,\ \psi(1,2) = +,\ \psi(1,3) = 0,\ \psi(2,3) = -\,.\]
  \end{ex}
  
Note that  \cite[Theorem~1.2]{GT} can be rephrased in our settings as the characterization of configurations of residues over which the isoresidual fiber is empty for any stratum of meromorphic $1$-forms. It appears that such configurations lie in the intersection of too many resonance hyperplanes (and therefore any decorated tree would have at least one edge corresponding to a vanishing partial sum of residues).\newline

To conclude on the topic of degeneration of $1$-forms, it is worth noting that it is  possible to make other degenerations of the $1$-forms. For example, it is to use the incidence variety compactification described in~\cite{BCGGM}.  This has the advantage to keep track of more information than the WYSIWYG but the drawback is that it does not give a direct relation with the tree.  We could use the theory of multi-scale differentials introduced in \cite{BCGGM2} to keep track of lots of information and obtain a smooth family. But in order to obtain this family, we should blowup the residual space along intersection of resonance hyperplanes. This makes the picture more complicated for our purpose, but could be useful to study other problems.
\newline

Finally we show that the degree of the cardinality of an isoresidual fiber only depends on the resonance hyperplanes it belongs to. Recall that we show in Proposition~\ref{prop:realgen}  that in a given family of resonance hyperplanes there exists a fiber with real residues having the same number of elements that any given fiber. In Lemma~\ref{lem:arbresdiff} and~\ref{lem:arbresdiffdegen} we show that the number of elements of such a real fiber only depends of its sign function. We now show that it does not depends on the sign function.

\begin{thm}\label{thm:dependarrangement}
For any stratum $\mathcal{H}$ of rational differentials with one zero, the number of elements of an isoresidual fiber $\mathcal{F}$ only depends on the resonance hyperplanes $\mathcal{F}$ belongs to.
\end{thm}

\begin{proof}
It follows from Proposition~\ref{prop:realgen}, Lemma~\ref{lem:arbresdiff} and~\ref{lem:arbresdiffdegen} that it suffices to prove that two real isoresidual fibers $\mathcal{F}$ and $\mathcal{F'}$ that belong to the same family of resonance hyperplanes but with different sign functions $\psi_{\mathcal{F}}$ and $\psi_{\mathcal{F'}}$ are in bijection. We denote by $B \subset \mathcal{R}_{p}$ the complex vector space given by the intersection of these resonance hyperplanes.  Its intersection $\mathbb{R}B$ with the space $\mathbb{R}\mathcal{R}_{p}$ of real configurations of residues is endowed with a hyperplane arrangement $\mathbb{R}\mathcal{A}_{p}^{B}$ given by the intersection with the resonance arrangement.  This hyperplane arrangement cuts out $\mathbb{R}B$ into chambers. Note that  the isoresidual fibers $\mathcal{F}$ and $\mathcal{F'}$ belong to different chambers.
\newline
It follows from Lemma~\ref{lem:arbresdiff} and~\ref{lem:arbresdiffdegen}  that the number of elements of the fibers is constant in any chamber of $\mathbb{R}B$. We will prove that for any pair of chambers $\mathcal{C}$ and $\mathcal{C'}$ separated by a hyperplane,  these fibers are in bijection. We consider  real fibers $\mathcal{F}$ in $\mathcal{C}$ and $\mathcal{F}'$ in $\mathcal{C}'$ that are closed to each other. Then there exists a complex fiber $\mathcal{G}$ over a point of the same cell which is related to both $\mathcal{F}$ and $\mathcal{F}'$ by the contraction flow. The same argument as in Proposition~\ref{prop:realgen}  using the fact that the residues form a local coordinate system leads to the bijection of $\mathcal{F}$ with $\mathcal{F'}$.
Since every chamber can be joined to another by crossing finitely many line of the real arrangement, the number of elements of the fibers in any chamber of $\mathbb{R}B$ is the same. Therefore, the number of elements of an isoresidual fiber $\mathcal{F}$ only depends on the resonance hyperplanes $\mathcal{F}$ belongs to.
\end{proof}

Note that Theorem~\ref{thm:dependarrangement} is in fact a geometric version of the following combinatorial result. Two sets of decorated trees with different sign functions but with the same zero-set have the same number of elements.

\section{Counting decorated trees}
\label{sec:compter}

The main goal of this section is to prove Theorem~\ref{thm:numerofibre} computing the number of elements of isoresidual fibers.
The computation in the non-resonant case is done in Section~\ref{sec:nonresonance} and in some resonant cases in Section~\ref{sec:comptereson}.
 Since we prove in Section~\ref{sec:transarbrees} that this computation reduces to the  enumeration of decorated trees compatible with a sign function, this computation is combinatorial.
%
%
%
%
%
%
%

\subsection{Non-resonant case}
\label{sec:nonresonance}

We begin this section by the case of strata with two poles.

\begin{prop}\label{prop:trivial2}
For any partition $(a;-b_{1},-b_{2})$ of $-2$  and for any non identically vanishing sign function $\psi$, there is a unique decorated tree in $\mathcal{T}(b_{1},b_{2},\psi)$.
\end{prop}

\begin{proof}
A tree with two vertices is just a edge joining two vertices. The orientation of the edge goes from the vertex $v_{1}$ such that $\psi(1)<0$ to the vertex $v_{2}$ such that $\psi(v_{2})>0$.
\end{proof}

Note that Proposition~\ref{prop:trivial2} gives an alternative proof of the fact that any stratum $\mathcal{H}(a;-b_{1},-b_{2})$ is isomorphic to $\mathbb{C^{\ast}}$. The residue map gives an isomorphism between the stratum $\mathcal{H}(a;-b_{1},-b_{2})$ and  $\mathcal{R}_{2} \setminus \mathcal{A}_{2}$.\newline

We now focus on the space $\mathcal{R}_{p} \setminus \mathcal{A}_{p}$, which corresponds to the sign functions that are everywhere different from zero.  Since this is the generic situation, this computation will give the degree of the isoresidual cover. \newline
By Theorem~\ref{thm:dependarrangement} we can focus on configurations of real residues such that one is positive and all other are negative. Indeed, the fibers over such configurations belong to the complement of any resonance hyperplane.\newline

In the following, we denote by $d(b_{1},\dots,b_{p})$ the number of elements of $\mathcal{T}(b_{1},\dots,b_{p},\psi_{p})$ where sign function $\psi_{p}$ is such that for any nonempty strict subset $I$ of $\{1,\dots,p\}$, the sign $\psi(I)$ is positive if and only if  $1 \in I$, and is negative otherwise.
The rest of the subsection is mainly combinatorial.\newline

We first prove an  induction relation between various $d$ which corresponds to the operation of adding one simple pole with negative residue in the stratum.

\begin{prop}\label{prop:recurencedegre}
For any  partition $(a;-b_{1},\dots,-b_{p})$ of $-2$, the following relation holds 
\[ d(b_{1},\dots,b_{p},1) = (a+1) \cdot d(b_{1},\dots,b_{p})\,. \] 
\end{prop}

We will illustrate the construction of this proof in Example~\ref{ex:arrangement4} and give its geometric interpretation in Remark~\ref{rem:recurencedegre}. Moreover, given a decorated tree, a {\em corner} is the interval of a vertex between two consecutive markings (edges or half-edges) at a vertex.  
\begin{proof}
Let us fix a decorated tree $\mathfrak{t}$ of $\mathcal{T}(b_{1},\dots,b_{p},\psi_{p})$. We compute the number of decorated trees of $\mathcal{T}(b_{1},\dots,b_{p},1,\psi_{p+1})$ that can be obtain from $\mathfrak{t}$ by adding the vertex corresponding to the new simple pole. Since $\psi_{p+1}(p+1)<0$, the vertex that we add is connected to the rest of the graph by an edge going from this vertex to the vertex of~$\mathfrak{t}$. There are $2a+2$ corners on~$\mathfrak{t}$. Consider a corner following an edge pointing into the corresponding vertex. By the last condition of Definition~\ref{def:abstractdecoratedtree}, we can put the new vertex at this corner. For the next corner in the graph (which may be on an other vertex), this condition does not allow us to put the new vertex. Indeed if these two corner are separated by a half-edge, there would be an odd number of half-edges between two edges pointing in the same direction. If the new corner is on a new vertex, the direction of the edge is from this new vertex to the old one, hence we need an odd number of half-edges between this edge and the edge that we want to add. Now continuing, we see that we can put the new vertex at half of the corners.   Therefore, there are $a+1$ places to glue the leaf corresponding to the additional simple pole. Hence we get $(a+1)\cdot d(b_{1},\dots,b_{p})$ different decorated trees in $\mathcal{T}(b_{1},\dots,b_{p},1,\psi_{p+1})$.\newline
Conversely, this surgery is invertible. Indeed take a decorated tree $\mathcal{T}(b_{1},\dots,b_{p},1,\psi_{p+1})$. We can erase the leaf corresponding to the simple pole and get the decorated tree of $\mathcal{T}(b_{1},\dots,b_{p},\psi_{p})$. This  proves the relation.
\end{proof}

\begin{ex}\label{ex:arrangement4}
 We continue Example~\ref{ex:arrangement} in the stratum $\mathcal{H}(4;-2,-2,-2)$. Consider the top differential pictured in the right of Figure~\ref{fig:arrangement1} and its associated graph pictured in the left of Figure~\ref{fig:arrangement2}. The five decorated trees obtained from this decorated tree by the operation of the proof of Proposition~\ref{prop:recurencedegre} are pictured in Figure~\ref{fig:arrangement4}. We obtain decorated trees corresponding to $1$-forms in the stratum $\mathcal{H}(4;-2,-2,-2,-1)$.
 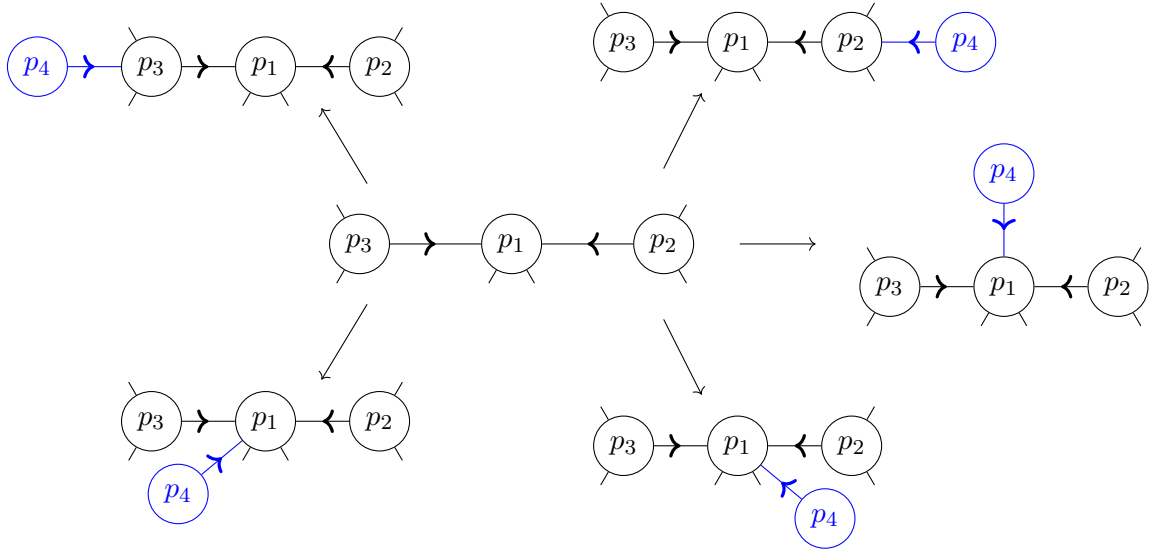
\begin{figure}[ht]
\begin{tikzpicture}[scale=1,decoration={
    markings,
    mark=at position 0.5 with {\arrow[very thick]{>}}}]
   
\node[circle,draw] (a) at  (-2,0) {$p_{3}$};
\node[circle,draw] (b) at  (0,0) {$p_{1}$};
\node[circle,draw] (c) at  (2,0) {$p_{2}$};
\draw [postaction={decorate}]  (a) --  (b);
\draw[postaction={decorate}] (c) -- (b);

\draw (a) -- ++(120:.6);
\draw (a) -- ++(240:.6);
\draw (c) -- ++(60:.6);
\draw(c) -- ++(-60:.6);
\draw(b) -- ++(-120:.6);
\draw (b) -- ++(-60:.6);

\draw[->] (3,0) -- (4,0);
\draw[->] (2,1) -- (2.5,2);
\draw[->] (-1.9,.8) -- (-2.5,1.8);
\draw[->] (2,-1) -- (2.5,-2);
\draw[->] (-1.9,-.8) -- (-2.5,-1.8);
\begin{scope}[shift=(-5:6.5)]
 \node[circle,draw] (a) at  (-1.5,0) {$p_{3}$};
\node[circle,draw] (b) at  (0,0) {$p_{1}$};
\node[circle,draw] (c) at  (1.5,0) {$p_{2}$};
\node[circle,draw,blue] (d) at  (0,1.5) {$p_{4}$};
\draw [postaction={decorate}]  (a) --  (b);
\draw[postaction={decorate}] (c) -- (b);
\draw[postaction={decorate},blue] (d) -- (b);

\draw (a) -- ++(120:.6);
\draw (a) -- ++(240:.6);
\draw (c) -- ++(60:.6);
\draw(c) -- ++(-60:.6);
\draw(b) -- ++(-120:.6);
\draw (b) -- ++(-60:.6);
\end{scope}

\begin{scope}[shift=(42:4)]
 \node[circle,draw] (a) at  (-1.5,0) {$p_{3}$};
\node[circle,draw] (b) at  (0,0) {$p_{1}$};
\node[circle,draw] (c) at  (1.5,0) {$p_{2}$};
\node[circle,draw,blue] (d) at  (3,0) {$p_{4}$};
\draw [postaction={decorate}]  (a) --  (b);
\draw[postaction={decorate}] (c) -- (b);
\draw[postaction={decorate},blue] (d) -- (c);

\draw (a) -- ++(120:.6);
\draw (a) -- ++(240:.6);
\draw (c) -- ++(60:.6);
\draw(c) -- ++(-60:.6);
\draw(b) -- ++(-120:.6);
\draw (b) -- ++(-60:.6);
\end{scope}

\begin{scope}[shift=(144:4)]
 \node[circle,draw] (a) at  (-1.5,0) {$p_{3}$};
\node[circle,draw] (b) at  (0,0) {$p_{1}$};
\node[circle,draw] (c) at  (1.5,0) {$p_{2}$};
\node[circle,draw,blue] (d) at  (-3,0) {$p_{4}$};
\draw [postaction={decorate}]  (a) --  (b);
\draw[postaction={decorate}] (c) -- (b);
\draw[postaction={decorate},blue] (d) -- (a);

\draw (a) -- ++(120:.6);
\draw (a) -- ++(240:.6);
\draw (c) -- ++(60:.6);
\draw(c) -- ++(-60:.6);
\draw(b) -- ++(-120:.6);
\draw (b) -- ++(-60:.6);
\end{scope}

\begin{scope}[shift=(-144:4)]
 \node[circle,draw] (a) at  (-1.5,0) {$p_{3}$};
\node[circle,draw] (b) at  (0,0) {$p_{1}$};
\node[circle,draw] (c) at  (1.5,0) {$p_{2}$};
\node[circle,draw,blue] (d) at  (-140:1.5) {$p_{4}$};
\draw [postaction={decorate}]  (a) --  (b);
\draw[postaction={decorate}] (c) -- (b);
\draw[postaction={decorate},blue] (d) -- (b);

\draw (a) -- ++(120:.6);
\draw (a) -- ++(240:.6);
\draw (c) -- ++(60:.6);
\draw(c) -- ++(-60:.6);
\draw(b) -- ++(-120:.6);
\draw (b) -- ++(-60:.6);
\end{scope}

\begin{scope}[shift=(-42:4)]
 \node[circle,draw] (a) at  (-1.5,0) {$p_{3}$};
\node[circle,draw] (b) at  (0,0) {$p_{1}$};
\node[circle,draw] (c) at  (1.5,0) {$p_{2}$};
\node[circle,draw,blue] (d) at  (-40:1.5) {$p_{4}$};
\draw [postaction={decorate}]  (a) --  (b);
\draw[postaction={decorate}] (c) -- (b);
\draw[postaction={decorate},blue] (d) -- (b);

\draw (a) -- ++(120:.6);
\draw (a) -- ++(240:.6);
\draw (c) -- ++(60:.6);
\draw(c) -- ++(-60:.6);
\draw(b) -- ++(-120:.6);
\draw (b) -- ++(-60:.6);
\end{scope}
 \end{tikzpicture}
 \caption{One of the decorated tree pictured in Figure~\ref{fig:arrangement2} and the operation of the proof of Proposition~\ref{prop:recurencedegre}} \label{fig:arrangement4}
\end{figure}
\end{ex}

Before to continue with the proof, we want to give the geometric interpretation of Proposition~\ref{prop:recurencedegre}.
\begin{rem}\label{rem:recurencedegre}
Given a $1$-form associated with a graph as in Proposition~\ref{prop:recurencedegre}. Adding a pole with a negative residue at the zero of order $a$ correspond to choose a positive prong at the zero and glue a flat cylinder along this prong (after deforming a bit the flat structure associated to this $1$-form). Hence the factor $(a+1)$ of Proposition~\ref{prop:recurencedegre} represents the prong-number of the zero of order $a$. This construction can be make precise using \cite{BCGGM2}. 
\end{rem}

We now prove that we can transfer some weight from one vertex to another without changing the number of decorated trees.

\begin{prop}\label{prop:changepoids}
For any  partition $(a;-b_{1},\dots,-b_{p}$) of $-2$, the number $d(b_{1},\dots,b_{p})$ only depends  on the number of poles $p$ and on the order~$a$.
\end{prop}

\begin{proof}
We consider the set of decorated trees $\mathcal{T}(b_{1},\dots,b_{p},\psi_{p})$ where $\psi_{p}$ is the standard sign function. We suppose that there is a label $i \geq 2$ such that $b_{i} \geq 2$. Note first that since the sign of every proper subset $I$ is positive if and only if  $1 \in I$, the decorated trees are rooted trees where the root is the vertex with the label $1$ and every edge is oriented in the direction of the root.\newline
 We now define an operation $$s\colon \mathcal{T}(b_{1},\dots,b_{p},\psi_{p}) \to \mathcal{T}(b_{1}+1,b_{2},\dots, b_{i-1}, b_{i}-1, b_{i+1} \dots,b_{p},\psi_{p})\,.$$ 
 In the first step, remove the two first half-edge of vertex $i$ next to the edge going to the root in the counterclockwise order and the whole portion~$P$ of tree between them. In the second step, glue in the first next corner on the root the first half-edge, then the portion~$P$ of the graph and finally the last half-edge. This operation is pictured in Figure~\ref{fig:operations}.
 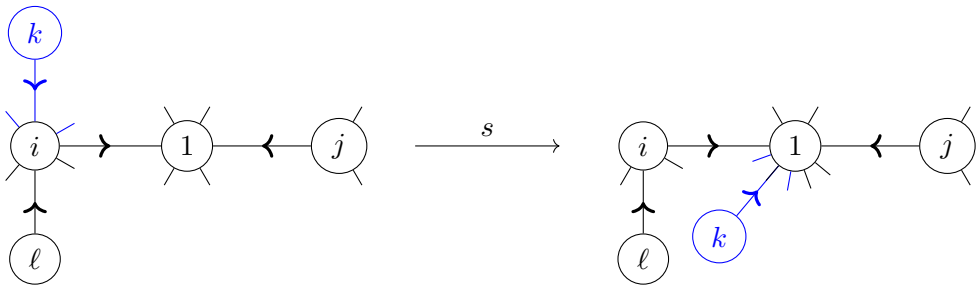
\begin{figure}[ht]
\begin{tikzpicture}[scale=1,decoration={
    markings,
    mark=at position 0.5 with {\arrow[very thick]{>}}}]
   
\node[circle,draw] (a) at  (-2,0) {$i$};
\node[circle,draw] (b) at  (0,0) {$1$};
\node[circle,draw] (c) at  (2,0) {$j$};
\node[circle,draw,blue] (d) at  (-2,1.5) {$k$};
\node[circle,draw] (e) at  (-2,-1.5) {$ \ell$};
\draw [postaction={decorate}]  (a) --  (b);
\draw[postaction={decorate}] (c) -- (b);
\draw [postaction={decorate},blue]  (d) --  (a);
\draw[postaction={decorate}] (e) -- (a);

\draw[blue] (a) -- ++(130:.6);
\draw (a) -- ++(230:.6);
\draw[blue] (a) -- ++(30:.6);
\draw (a) -- ++(-30:.6);
\draw (c) -- ++(60:.6);
\draw(c) -- ++(-60:.6);
\draw(b) -- ++(-120:.6);
\draw (b) -- ++(-60:.6);
\draw(b) -- ++(120:.6);
\draw (b) -- ++(60:.6);

\draw[->] (3,0) -- node[above] {$s$} (4.9,0);

\begin{scope}[xshift=8cm]
\node[circle,draw] (a) at  (-2,0) {$i$};
\node[circle,draw] (b) at  (0,0) {$1$};
\node[circle,draw] (c) at  (2,0) {$j$};
\node[circle,draw,blue] (d) at  (-1,-1.2) {$k$};
\node[circle,draw] (e) at  (-2,-1.5) {$ \ell$};
\draw [postaction={decorate}]  (a) --  (b);
\draw[postaction={decorate}] (c) -- (b);
\draw [postaction={decorate},blue]  (d) --  (b);
\draw[postaction={decorate}] (e) -- (a);

\draw[blue] (b) -- ++(-160:.6);
\draw (b) -- ++(230:.6);
\draw[blue] (b) -- ++(-100:.6);
\draw (a) -- ++(-30:.6);
\draw (c) -- ++(60:.6);
\draw(c) -- ++(-60:.6);
\draw(b) -- ++(-70:.6);
\draw (b) -- ++(-40:.6);
\draw(b) -- ++(120:.6);
\draw (b) -- ++(60:.6);
\draw(a) -- ++(-120:.6);
\end{scope}

 \end{tikzpicture}
 \caption{The operation $s$ on some decorated tree with the moved part in blue} \label{fig:operations}
\end{figure}

The operation $s$ is clearly invertible. Indeed, we just have to consider the first corners of the root after the vertex $i$ and to move back to $i$ the first two half-edges and the portion of the graph between them. 
Thus, the operation~$s$ gives a bijection between the set of decorated trees $\mathcal{T}(b_{1},\dots,b_{p},\psi_{p})$ to the set $\mathcal{T}(b_{1}+1,b_{2},\dots, b_{i-1}, b_{i}-1, b_{i+1} \dots,b_{p},\psi_{p})$.\newline
By performing the operation $s$ a finite number of steps, we get a bijection between the set of decorated trees of any $\mathcal{T}(b_{1},\dots,b_{p},\psi_{p})$ to the set $\mathcal{T}(a+3-p,1,\dots,1,\psi_{p})$.
\end{proof}

The latter proposition reduces the computation of the degree of the isoresidual cover to the case of strata with only one pole of order greater or equal to~$2$ and several simple poles. We now prove Theorem~\ref{thm:numerofibre} by induction on the number of simple poles.

\begin{proof}[Proof of Theorem \ref{thm:numerofibre}]
According to Proposition~\ref{prop:changepoids} the number of decorated trees in $\mathcal{T}(b_{1},\dots,b_{p},\psi_{p})$ is equal to the set $\mathcal{T}(a+3-p,1,\dots,1,\psi_{p})$, where $a=-2 + \sum_{i} b_{i}$ and there are $p-1$ simple poles. The enumeration of the elements in $\mathcal{T}(a+3-p,1,\dots,1,\psi_{p})$ is obtained by induction on the number of simple pole. According to Proposition~\ref{prop:trivial2}, if $p=2$ there is a unique element in $\mathcal{T}(a+1,1,\psi_{p})$. Now  using Proposition~\ref{prop:recurencedegre} we get that for every $p\geq3$ the following equality holds  
$$ d(a+3-p,1,\dots,1) =  d(a+1,1) \cdot\prod_{i=0}^{2-p}(a+i) \,.$$ This implies that for any $p \geq 2$ the following equation holds  $$d(a+3-p,1,\dots,1)=\frac{a!}{(a+2-p)!} \,.$$ 

Finally, according to Lemma~\ref{lem:arbresdiff} and Theorem~\ref{thm:dependarrangement}, the isoresidual fiber in the stratum $\mathcal{H}(a,-b_{1},\dots,-b_{p})$ over the non-resonant space $\mathcal{R}_{p} \setminus \mathcal{A}_{p}$ is in bijection with the set of decorated trees of $\mathcal{T}(b_{1},\dots,b_{p},\psi_{p})$. Therefore, these generic isoresidual fibers have $\frac{a!}{(a+2-p)!}$ elements. This is the degree of the isoresidual cover.
\end{proof}

\subsection{Resonant case}\label{sec:comptereson}

We first treat the cases which lie in the intersection of many hyperplanes. The first part of the proposition is well-know, while the second part was proven in Proposition~2.3 of \cite{CC}. Even if our proof is essentially the same as the original, our method allows to clarify the original proof.

\begin{prop}\label{prop:trivial}
For any  stratum $\mathcal{H} (a;-b_{1},\dots,-b_{p})$ with $p\geq2$
\begin{enumerate}
 \item the isoresidual fiber at the origin is empty,
 \item the degree above the intersection of the $p-2$ hyperplanes $\lambda_{i} = 0$ for $i\in \{2,\dots,p-1\}$ is 
 $$ (p-2)! \cdot \prod_{i=2}^{p-1} (b_{i}-1) \,.$$
\end{enumerate}
\end{prop}

Note that any locus which is given by the intersection of $p-2$ hyperplanes can be obtained by intersection of hyperplanes of the form $r_{i} = 0$. Hence the second part of Proposition~\ref{prop:trivial} gives the degree above all these loci.

\begin{proof}
Given the  partition $(a;-b_{1},\dots,-b_{p})$ of $-2$ with $p\geq2$, the first part of the proposition is equivalent to the fact that the set of decorated trees of type $\mu$ compatible with the identically zero sign function is empty. Note that a  decorated tree is connected, hence has edges since $p\geq2$. Thus the sign of at least one subset of the labels should be nonzero.

For the second part, using Lemma~\ref{lem:corWYSIWYG} and~\ref{lem:arbresdiffdegen}, it suffices to count the number of decorated trees compatible with the sign function $\psi$ such that $\psi(1) = +$ and $\psi(i)=0$ for $i\in \{2,\dots,p-1\}$ such that there are no edge separating a vertex indexed by $i$ to the rest of the graph. This implies that $\mathfrak{t}$ is a chain connecting the vertex $1$ to the vertex $p$. The position of the other labels is arbitrary, which gives $|S_{p-2}|$ possibilities. Moreover, at each vertex of label $i\in \{2,\dots,p-1\}$, there are $b_{i}-1$ possible positions for the half-edges.
\end{proof}

We now compute the number of elements of isoresidual fibers that belong to exactly one resonance hyperplane. This will be extensively use in Section~\ref{sec:monodromie} to understand the monodromy of the isoresidual fibration.  Such a computation for fibers in the intersection of several hyperplanes is an essentially algorithmic problem for which there does not appear to be a general formula.\newline

In the following, we consider a stratum $\mathcal{H}(a,-b_{1},\dots,-b_{p})$ with $\sum_{j=1}^{p} b_{j} = a+2$ and an isoresidual fiber $\mathcal{F}$ that belongs to exactly one resonance hyperplane $A_{I}$ corresponding to a partition $I \cup I^{\complement}$ of $\{1,\dots,p\}$. Recall from Definition~\ref{def:resonanceintro} that the hyperplane $A_{I}$ is given by the equation $\sum_{j \in I} \lambda_{j} = 0$ in the residual space.\newline

As in the generic case, we use Theorem~\ref{thm:dependarrangement} to reduce the problem to the case of real fibers. Moreover, according to Lemma~\ref{lem:arbresdiffdegen}, given a subset $I$ of $\{1,\dots,p\}$ we need to understand the number of distinct decorated trees which have an edge cutting them into two trees with vertices labeled by $I$ and $I^{\complement}$. The enumeration of decorated trees corresponding to elements of the fibers involves the number of poles and the resonance degree of each part of the partition. Recall from Definition~\ref{def:resonanceintro} that given a partition $\mu=(a;-b_{1},\dots,-b_{p})$ of~$-2$ and a subset $I$ of $\{1,\dots,p\}$, the  resonance degree is   $d_{I}=-1+\sum_{j \in I} b_{j}$ and $c_{I}$ is the cardinal of~$I$.

\begin{prop}\label{prop:distributiondeg}
For any  partition $\mu=(a;-b_{1},\dots,-b_{p})$ of~$-2$, any sign function $\psi$ which is nonzero for any nonempty strict subset of $\{1,\dots,p\}$ and any nontrivial partition $I \cup I^{\complement}$ of $\{1,\dots,p\}$, there are 
$$\frac{d_{I}!\, (a-d_{I})!}{(d_{I}+1-c_{I})! \, (a+1-p+c_{I}-d_{I})!}$$ decorated trees in $\mathcal{T}(b_{1},\dots,b_{p},\psi)$ in which an edge cuts out the tree according to the partition $I \cup I^{\complement}$ of the vertices.
\smallskip
\par
Moreover, if neither $I$ or $I^{\complement}$ are singletons, these decorated trees are divided into 
$$\frac{(d_{I}-1)!\, (a-d_{I}-1)!}{(d_{I}+1-c_{I})! \, (a+1-p+c_{I}-d_{I})!}$$ classes of $d_{I}(a-d_{I})$ elements in which the subtrees corresponding to $I$ and $I^{\complement}$ coincide.\newline
If $I^{\complement}$ is a singleton and $p \geq 3$, the $\frac{d_{I}!}{(d_{I}+2-p)!}$ decorated trees are divided into $\frac{(d_{I}-1)!}{(d_{I}+2-p)!}$ classes of $d_{I}$ elements.\newline
If $p=2$ and both $I$ and $I^{\complement}$ are singletons, there is only one decorated tree. 
\end{prop}

\begin{proof}
Let us consider a decorated tree $\mathfrak{t}$ in $\mathcal{T}(b_{1},\dots,b_{p},\psi_{p})$ that has an edge dividing the set of vertices of $\mathfrak{t}$ into $I$ and $I^{\complement}$. Let us delete this edge to obtain two trees $\mathfrak{t}_{I}$ and $\mathfrak{t}_{I^{\complement}}$. For the subtree $\mathfrak{t}_{I}$, we replace the deleted edge by a leaf with a vertex of degree one to make a new tree $\tilde{\mathfrak{t}}_{I}$. The orientation of this new edge is the same as the one of the deleted edge. In this way,  $\tilde{\mathfrak{t}}_{I}$ is a decorated tree. We define a sign function $\psi_{I}$ which coincide with~$ \psi$ on the subsets of $I$ and for the subsets containing the new label $\psi_{I}$ coincide with the sign of $\psi$ on the set where we replace this label by $I^{\complement}$ . We define similarly a tree $\tilde{\mathfrak{t}}_{I^{\complement}}$ starting from $\mathfrak{t}_{I^{\complement}}$. Note that the sign function restricted to both $\tilde{\mathfrak{t}}_{I}$  and $\tilde{\mathfrak{t}}_{I^{\complement}}$  does not vanish on any nonempty proper subset of the set of its vertices. Hence according to Theorem~\ref{thm:numerofibre}, there are $\frac{d_{I}!}{(d_{I}+1-c_{I})!}$ of such~$\tilde{\mathfrak{t}}_{I}$.\newline
Reciprocally, we can choose both $\tilde{\mathfrak{t}}_{I}$  and $\tilde{\mathfrak{t}}_{I^{\complement}}$ independently of each other.  From these decorated tree, we can recover the tree $\mathfrak{t}$ in a unique way. Hence the number of decorated trees is given by multiplying the number of decorated trees   $\tilde{\mathfrak{t}}_{I}$  with the number of decorated trees~$\tilde{\mathfrak{t}}_{I^{\complement}}$. This proves the first claim of the proposition.
\smallskip
\par
We now treat the second part of the proposition.
In both of $\tilde{\mathfrak{t}}_{I}$  and $\tilde{\mathfrak{t}}_{I^{\complement}}$, we erase the leaf corresponding to the special leaf. By doing this operation we loose the information of the location of this leaf. If the resonance degree of $I$ is $d_{I}$, there are $d_{I}$ corners from which we could remove the leaf. This gives a total of $d_{I}(a-d_{I})$ choices, up to isotopy in the sphere, once we fix the shape of each subtree where  the leaf has been removed. If $I$ or~$I^{\complement}$ is a singleton, the corner on which the leaf is glued does not matter since in this case we have an exceptional symmetry. 
\end{proof}

 Note that in the setting of the proposition, if neither $I$ nor $I^{\complement}$ are singletons, there are $d_{I}(a-d_{I})$ decorated trees which  correspond to the same element in WYSIWYG. The resonance hyperplane belongs to the ramification of locus of the isoresidual cover extended to the WYSIWYG of the stratum.  This will be illustrate in Example~\ref{ex:arrangement5}.

We now compute the degree of the isoresidual fibration above a unique resonance hyperplane, thus proving Proposition~\ref{prop:degintro}.
\begin{cor}\label{cor:elementresonance}
The number of elements of an isoresidual fiber $\mathcal{F}$ that belongs to exactly one resonance hyperplane $A_{I}$ is $$\frac{a!}{(a+2-p)!}-\frac{d_{I}!\, (a-d_{I})!}{(d_{I}+1-c_{I})! \, (a+1-p+c_{I}-d_{I})!}\,.$$
\end{cor}

\begin{proof}
Without loss of generality, we assume by Theorem~\ref{thm:dependarrangement} that $\mathcal{F}$ is a real fiber. The sign function $\psi_{\mathcal{F}}$ is zero only for subsets $I$ and~$I^{\complement}$. The elements of $\mathcal{F}$ are classified by the set of decorated trees in $\mathcal{T}(b_{1},\dots,b_{p},\psi_{\mathcal{F}})$.\newline
The fiber $\mathcal{F}$ is in the closure of some chamber $\mathcal{C}$ of the real arrangement in the real part of the residual space. The elements of the fibers over the chamber $\mathcal{C}$ are classified by the set of decorated trees in $\mathcal{T}(b_{1},\dots,b_{p},\psi_{\mathcal{C}})$, where the sign of any nonempty proper subset of $\{1,\dots,p\}$ is nonzero. Without loss of generality, we assume $\psi_{\mathcal{C}}(I)$ is positive and $\psi_{\mathcal{C}}(I^{\complement})$ is negative. According to Theorem~\ref{thm:numerofibre},  there are $\frac{a!}{(a+2-p)!}$ such trees. \newline
The decorated trees compatible with $\psi_{\mathcal{F}}$ are exactly the same that are compatible with~$\psi_{\mathcal{C}}$ and for which there exists no edge that splits the set of vertices of the decorated tree  according to the partition $I \cup I^{\complement}$. By Proposition~\ref{prop:distributiondeg}, among the decorated trees compatible with $\psi_{\mathcal{C}}$ there are exactly $\frac{d_{I}!\, (a-d_{I})!}{(d_{I}+1-c_{I})!\, (a+1-p+c_{I}-d_{I})!}$ of them having such an edge.
\end{proof}
It should be noted that above the resonance hyperplane corresponding to the vanishing of the residue of only one pole, the number of elements in the fiber is simply given by the formula $\frac{a!}{(a+2-p)!}-\frac{d_{I}!}{(d_{I}+2-p)!}$. 

We finally illustrate the results of this section on an example.
\begin{ex}\label{ex:arrangement5}
 Let us consider the case of $\mathcal{H}(4,-2,-2,-2)$ introduced in Example~\ref{ex:arrangement}. Take the partition $\{2\}\cup\{1,3\}$ of the poles. On the $1$-form pictured in the top of Figure~\ref{fig:arrangement1}, we have the saddle connection labeled by $2$ which separated the poles in this partition. On the other hand, there is no such saddle connection on the lower $1$-form of this figure. But in the $1$-form where we permute the role of $p_{2}$ and $p_{3}$, there is such saddle connection. Hence the $1$-form represented in Figure~\ref{fig:arrangement3} is the unique one above the resonance hyperplane~$A_{1,2}$ and compatible with the considered sign function. The other $3$ elements degenerate to a singular element of $\bar{\mathcal{H}}(4,-2,-2,-2)$. This coincide with the computation of Corollary~\ref{cor:elementresonance}. Indeed in this case $d_{I}=1$ and $c_{I}=1$, hence the formula is 
 $$ \frac{4!}{3!} - \frac{1!\,3!}{1!\, 0!} = 1\,.$$
\end{ex}

\section{Monodromy of the isoresidual cover}
\label{sec:monodromie}

We know by Theorem \ref{thm:numerofibre}  that the isoresidual fibration of a stratum $\mathcal{H}(a,-b_{1},\dots,-b_{p})$ is an unramified cover of degree $\frac{a!}{(a+2-p)!}$ over the complement $\mathcal{R}_{p} \setminus \mathcal{A}_{p}$ of the resonance arrangement in the residual set. In this section we study the monodromy of this cover.
\smallskip
\par

Let $W_{p}$  be the fundamental group of $\mathcal{R}_{p} \setminus \mathcal{A}_{p}$. For any stratum $\mathcal{H}(a,-b_{1},\dots,-b_{p})$ on the Riemann sphere, the {\em monodromy group $\mathcal{M}_{\mathcal{H}}$} of the isoresidual cover is the morphism from the fundamental group $W_{p}$ of $\mathcal{R}_{p} \setminus \mathcal{A}_{p}$ into the automorphism group of the fiber given by the lifting of the loops.  


To the best of our knowledge, little is known about groups $W_{p}$ for $p\geq4$ associated to the resonance arrangement. Indeed, we have shown in Lemma~\ref{lem:passimplice} that these arrangements are not simplicial.
On the contrary, these resonance arrangements are well understood if $p \leq 3$. The case $p=2$ is already settled by Proposition~\ref{prop:trivial2}. Indeed, the arrangement complement $\mathcal{R}_{2} \setminus \mathcal{A}_{2}$ is isomorphic to $\mathbb{C}^{\ast}$, the fundamental group $W_{2} $ is isomorphic to $ \mathbb{Z}$ and the isoresidual cover is an isomorphism. Hence the monodromy is trivial. 
In the case $p=3$,  the arrangement $\mathcal{A}_{3}$ is formed by three complex lines with trivial mutual intersection in a complex plane as shown in Figure~\ref{fig:arrangement}.  In particular, $W_{3} $ is isomorphic to the pure braid group~$ PB_{3}$ and its monodromy is computed in Sections~\ref{sec:monod2sim} and~\ref{sec:monod3poles}, proving Theorem~\ref{thm:monothree}.
\newline

In a fundamental group $W_{p}$, the oriented simple loops $\gamma_{I}$ around each resonance hyperplane $A_{I}$ form a set of generators. The loop $\gamma_{I}$ is oriented in such a way that the induced loop given by $\sum_{j \in I} \lambda_{j}$ in $\mathbb{C}^{\ast}$ makes a positive loop around zero.  The equality $\sum_{j \in I^{\complement}} \lambda_{j}=-\sum_{j \in I} \lambda_{j}$ implies that the loop associated to $I^{\complement}$ also makes a positive loop around zero. Hence the orientation is independent of the choice of $I$ or its complement. By abuse of notation, we will denote by $\gamma_{I}$ both the loop and its monodromy action on any fiber of $\mathcal{R}_{p} \setminus \mathcal{A}_{p}$. Moreover, if $I = \{r\}$ is a singleton, we write $\gamma_{r}$ instead of $\gamma_{\{r\}}$.\newline

In Section~\ref{sec:monodromiegen}, we give the monodromy of the generator $\gamma_{I}$ and in Section~\ref{sec:monodromierel} we give the relations between these elements. In Section~\ref{sec:monod2sim}, we apply this to study the strata with two simple poles.
In our strategy, we focus on real fibers that belong to a particular real chamber. Indeed, these elements are classified by the decorated trees compatible with a sign function that vanishes on no nontrivial subset and we study the action on these trees.\newline

\subsection{Monodromy of the generators}
\label{sec:monodromiegen}

We first consider the monodromy of simple loops around a resonance hyperplane. They form a set of generators of the fundamental group~$W_{p}$ of the complement of the resonance arrangement. Here, we consider the action of a loop~$\gamma_{I}$ on real fibers over a chamber $\mathcal{C}$ incident to hyperplane $A_{I}$.\newline

Monodromy along loops around resonance hyperplane defined by the vanishing of only one residue is a specific case. Recall that   $d_{I}=-1+\sum_{j \in I} b_{j}$ is the resonance degree of~$I$.

\begin{prop}\label{prop:monodromieloop}
If $I^{\complement}$ is a singleton, the monodromy action of loop $\gamma_{I}$ on any generic fiber stabilizes $\frac{a!}{(a+2-p)!}-\frac{d_{I}!}{(d_{I}+2-p)!}$ elements and has $\frac{(d_{I}-1)!}{(d_{I}+2-p)!}$ orbits of order $d_{I}$.
\end{prop}

\begin{proof}
We consider a chamber $\mathcal{C}$ of real fibers such that the hyperplane $A_{I}$ is adjacent to $\mathcal{C}$. We choose in the chamber a fiber $\mathcal{F}$ such that $\sum_{j \in I} \lambda_{j}$ is very small in modulus relatively to other partial sums of residues. Thus, the monodromy of the isoresidual cover along the loop $\gamma_{I}$ only affects edges cutting the decorated trees according to partition $I \cup I^{\complement}$. According to Proposition \ref{prop:distributiondeg}, there are $\frac{a!}{(a+2-p)!}-\frac{d_{I}!}{(d_{I}+2-p)!}$ elements of $\mathcal{F}$ that do not have such an edge and thus are fixed by the monodromy along~$\gamma_{I}$.\newline
Among the $\frac{d_{I}!}{(d_{I}+2-p)!}$ other elements, the loop $\gamma_{I}$ preserves every edge different from that cutting the trees according to partition $I \cup I^{\complement}$. Therefore, it preserves the shape of the two subtrees corresponding to $I$ and $I^{\complement}$ up to isotopy.  Then Proposition~\ref{prop:distributiondeg} implies that the action of $\gamma_{I}$ preserves $\frac{(d_{I}-1)!}{(d_{I}+2-p)!}$ classes of $d_{I}$ trees that differ only be the corner on which the leaf corresponding to the pole of the singleton $I^{\complement}$ is attached.\newline
For each loop around hyperplane $A_{I}$, the leaf moves in each of the two subtrees from one corner to the next allowed corner in the clockwise order. This move is trivial for the tree with only one vertex. After $d_{I}$ loops, we recover the initial tree. 
\end{proof}
Before going to the general case, we illustrate this result in our example of the stratum $\mathcal{H}(4,-2,-2,-2)$ started in Example~\ref{ex:arrangement}.
\begin{ex}\label{ex:arrangement6}
 Consider a loop that goes from the point pictured in Figure~\ref{fig:arrangement} around the hyperplane $A_{2}$. Then the monodromy is given in Figure~\ref{fig:arrangement6}. Indeed, the monodromy consists of making the segment labeled by $2$ small and rotation in the positive direction. 
 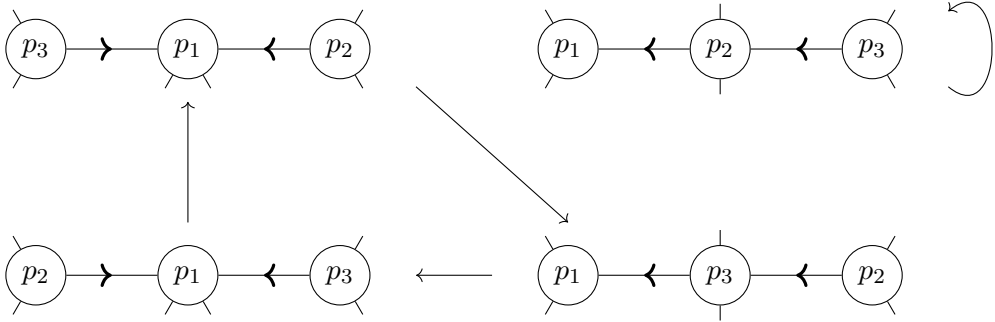
\begin{figure}[ht]
\begin{tikzpicture}[scale=1,decoration={
    markings,
    mark=at position 0.5 with {\arrow[very thick]{>}}}]
   
\node[circle,draw] (a) at  (-2,0) {$p_{3}$};
\node[circle,draw] (b) at  (0,0) {$p_{1}$};
\node[circle,draw] (c) at  (2,0) {$p_{2}$};
\draw [postaction={decorate}]  (a) --  (b);
\draw[postaction={decorate}] (c) -- (b);

\draw (a) -- ++(120:.6);
\draw (a) -- ++(240:.6);
\draw (c) -- ++(60:.6);
\draw(c) -- ++(-60:.6);
\draw(b) -- ++(-120:.6);
\draw (b) -- ++(-60:.6);

 \draw[->] (3,-.5) -- (5,-2.3);

\begin{scope}[xshift=7cm]
\node[circle,draw] (a) at  (-2,0) {$p_{1}$};
\node[circle,draw] (b) at  (0,0) {$p_{2}$};
\node[circle,draw] (c) at  (2,0) {$p_{3}$};
\draw [postaction={decorate}]  (b) --  (a);
\draw[postaction={decorate}] (c) -- (b);

 \draw (a) -- ++(120:.6);
\draw (a) -- ++(240:.6);
\draw (c) -- ++(60:.6);
\draw (c) -- ++(-60:.6);
\draw (b) -- ++(90:.6);
\draw (b) -- ++(-90:.6);

 \draw[->] (3,-.5) .. controls ++(-40:1) and ++(40:1) .. (3,.5) ;
\end{scope}

\begin{scope}[yshift=-3cm]
\node[circle,draw] (a) at  (-2,0) {$p_{2}$};
\node[circle,draw] (b) at  (0,0) {$p_{1}$};
\node[circle,draw] (c) at  (2,0) {$p_{3}$};
\draw [postaction={decorate}]  (a) --  (b);
\draw[postaction={decorate}] (c) -- (b);

\draw (a) -- ++(120:.6);
\draw (a) -- ++(240:.6);
\draw (c) -- ++(60:.6);
\draw(c) -- ++(-60:.6);
\draw(b) -- ++(-120:.6);
\draw (b) -- ++(-60:.6);

  \draw[->] (0,.7) -- (0,2.3);
\end{scope}

\begin{scope}[xshift=7cm,yshift=-3cm]
\node[circle,draw] (a) at  (-2,0) {$p_{1}$};
\node[circle,draw] (b) at  (0,0) {$p_{3}$};
\node[circle,draw] (c) at  (2,0) {$p_{2}$};
\draw [postaction={decorate}]  (b) --  (a);
\draw[postaction={decorate}] (c) -- (b);

 \draw (a) -- ++(120:.6);
\draw (a) -- ++(240:.6);
\draw (c) -- ++(60:.6);
\draw (c) -- ++(-60:.6);
\draw (b) -- ++(90:.6);
\draw (b) -- ++(-90:.6);

  \draw[->] (-3,0) -- (-4,0);
\end{scope}
 \end{tikzpicture}
 \caption{The monodromy associated to $\gamma_{2}$ on the set of decorated trees associated to the differentials of Figure~\ref{fig:arrangement1}} \label{fig:arrangement6}
\end{figure}
\end{ex}

The formula of the general case is slightly more complicated. We should think of the monodromy action as the rotation of two gears on each other. Common divisors between the number of teeth of the two gears lead to a smaller number of mutually accessible configurations by rotating the gears simultaneously. Note that this is similar to the prong-matching equivalence classes introduced in \cite{BCGGM2}.

\begin{prop}\label{prop:monodromieloopbis}
If neither $I$ nor $I^{\complement}$ is a singleton, then the monodromy action of loop~$\gamma_{I}$ stabilizes $\frac{a!}{(a+2-p)!}-\frac{d_{I}!\, (a-d_{I})!}{(d_{I}+1-c_{I})! \, (a+1-p+c_{I}-d_{I})!}$ elements and has $\frac{(d_{I}-1)! \, (a-d_{I}-1)!\, \gcd(d_{I},a-d_{I})}{(d_{I}+1-c_{I})! \, (a+1-p+c_{I}-d_{I})!}$ orbits of order $\lcm(d_{I},a-d_{I})$.
\end{prop}

\begin{proof}
The proof is similar to that Proposition~\ref{prop:monodromieloop}. In every class of $d_{I}(a-d_{I})$ elements in which the shape of the two subtrees corresponding to the partition of vertices $I \cup I^{\complement}$ is the same, the action of the loop moves the edge connecting the two subtrees to the next allowed corner in the clockwise order. Hence the study of orbits of the monodromy action reduces to those of function $$\varphi \colon \mathbb{Z}_{d_{I}} \times \mathbb{Z}_{a-d_{I}} : (x,y) \mapsto (x+1,y+1) \,.$$
It is well known  that $\varphi$ has $\gcd(d_{I},a-d_{I})$ orbits, each of cardinal $\lcm(d_{I},a-d_{I})$.
\end{proof}

As a consequence, we deduce that in some specific cases, the monodromy of a simple loop around a resonance hyperplane is trivial.
\begin{cor}\label{cor:trivmon}
The monodromy of a loop $\gamma_{I}$ on the generic fiber of a stratum $\mathcal{H}$ is trivial if and only if one the following conditions holds:
\begin{enumerate}[(i)]
 \item the stratum $\mathcal{H}$ is $\mathcal{H}(2,-1,-1,-1,-1)$ and $I$ contains two elements;
 \item the stratum $\mathcal{H}$ is $\mathcal{H}(a,-a,-1,-1)$ with $a \geq 1$ and $I=\{1\}$.
\end{enumerate}
\end{cor}

Now we deduce from Propositions~\ref{prop:monodromieloop} and~\ref{prop:monodromieloopbis} that  the monodromy of the isoresidual cover does not always coincide with the full symmetric group~$\mathfrak{S}(\mathcal{F})$ of the generic fiber~$\mathcal{F}$. Note however that the fact that  strata of meromorphic $1$-forms are connected in genus zero implies that the monodromy group is a transitive subgroup of the symmetric group.

\begin{cor}\label{cor:alternate}
Given a stratum $\mathcal{H}(a,-b_{1},\dots,-b_{p})$, the monodromy group of the isoresidual cover is contained in the alternating group $\mathfrak{A}(\mathcal{F})$ of the generic fiber $\mathcal{F}$ if and only if one of the following conditions is satisfied:
\begin{enumerate}[(i)]
 \item $p=3$ and $b_{1},b_{2},b_{3}$ have the same parity;
 \item $p=4$ and $b_{1},b_{2},b_{3},b_{4}$ are even;
 \item $p=5$ and $a$ is even or $b_{1},b_{2},b_{3},b_{4},b_{5}$ are odd;
 \item $p \geq 6$.
\end{enumerate}
\end{cor}

Recall for the proof that $c_{I}$ denotes the cardinal of the set~$I$.
\begin{proof}
Propositions~\ref{prop:monodromieloop} and~\ref{prop:monodromieloopbis} provide the cycle decomposition of a set of generators of the monodromy group. If $I^{\complement}$ is a singleton, $\gamma_{I}$ is the product of $\frac{(d_{I}-1)!}{(d_{I}+2-p)!}$ disjoint cycles of order $d_{I}$. Consequently, the signature of $\gamma_{I}(\mathcal{F})$ is negative if and only if $d_{I}$ is even and $p \leq 4$. Let $b_{p}$ be the degree of the pole of $I^{\complement}$, we have $d_{I}=a+1-b_{p}$. Hence $d_{I}$ is even if and only if $a$ and $b_{p}$ do not have the same parity.\newline
Similarly, if neither $I$ nor $I^{\complement}$ are singletons, $\gamma_{I}$ is the product of $\frac{(d_{I}-1)! \, (a-d_{I}-1)!\, \gcd(d_{I},a-d_{I})}{(d_{I}+1-c_{I})! \, (a+1-p+c_{I}-d_{I})!}$ disjoint cycles of order $\lcm(d_{I},a-d_{I})$. Therefore, the signature of $\gamma_{I}$ is positive if $d_{I}$ and $a-d_{I}$ are both odd, since in this case cycles are of odd order. It is also positive if $d_{I}$ and $a-d_{I}$ are both even, since there is an even number of cycles of the same order. Thus, such permutations always have positive signature if $a$ is even.\newline

If $a$ is odd, $\gamma_{I}$ is a product of disjoint cycles of even order. If $c_{I} \geq 4$ or $p-c_{I} \geq 4$, the quotient of factorials that gives the number of these cycles is always even and $\gamma_{I}$ has positive signature. Besides, $\gcd(d_{I},a-d_{I})$ is automatically an odd number and $\lcm(d_{I},a-d_{I})$ is an even number. If $p=4$ and $c_{I}=2$, there are $\gcd(d_{I},a-d_{I})$ cycles of order $\lcm(d_{I},a-d_{I})$. Thus, $\gamma_{I}$ has negative signature in this case. If $p=5$ and $c_{I}=3$, the number of cycles is $(d_{I}-1)\gcd(d_{I},a-d_{I})$. Thus, $\gamma_{I}$ has positive signature if and only if $d_{I}$ is odd. Similarly, if $c_{I}=2$, then the number of cycles is $(a-d_{I}-1)\gcd(d_{I},a-d_{I})$ and $\gamma_{I}$ has positive signature if and only if $d_{I}$ is even. Finally, if $p=6$ and $c_{I}=3$, the number of cycles is a $(d_{I}-1)(a-d_{I}-1)\gcd(d_{I},a-d_{I})$. The two first factors have distinct parities so $\gamma_{I}$ always has positive signature.\newline

Then, for every stratum we can decide if there is a generator with negative signature. If $p=3$, every nontrivial partition displays a singleton. If $a$ is odd, a generator is of negative signature if and only if the corresponding pole is of even degree. Therefore, the monodromy group is not contained in the alternating group of the fiber if and only if at least one pole is of even degree while $a$ is odd. If the three poles are of odd degree, every generator has positive signature.\newline
If $p=3$ and $a$ is even, a generator is of negative signature if and only if the corresponding pole is of odd degree. Therefore, the monodromy group is contained in the alternating group of the fiber if and only if every pole is of even degree.\newline

If $p=4$ and $a$ is odd, we already found a generator of negative signature. If $a$ is even, the only generators that may have a negative signature display a singleton in the partition. Such generators have a negative signature if and only if the corresponding pole has odd degree.\newline
If $p \geq 5$ and $a$ is even, every generator has positive signature.\newline
If $p=5$ and $a$ is odd, the generator corresponding to a subset $I$ of degree $d_{I}$ with two elements has positive signature if and only if $d_{I}$ is even. If $I$ has three elements, then this happens if and only if $d_{I}$ is odd. If the five poles are of odd order, then every generator has positive signature. Otherwise, it is always possible to get a subset $I$ of two elements such that $d_{I}$ is odd.\newline
If $p \geq 6$ and $a$ is odd, every generator has positive signature.
\end{proof}

Finally note that Propositions~\ref{prop:monodromieloop} and~\ref{prop:monodromieloopbis} can be used to get some informations on~$W_{p}$. Since the action by monodromy of loops $\gamma_{I}$ on some elements of some fibers has periodic points of arbitrarily high order, they cannot be torsion elements. This result can be deduced from the study of the Orlik-Solomon algebra of the arrangement, which is a combinatorial structure isomorphic to the cohomology of the complement of the arrangement as explained in Chapter~3 of \cite{OT}.\newline

\subsection{Commutations relations}\label{sec:monodromierel}

In this section, we compute the commutator of two elements $\gamma_{I}$ and $\gamma_{J}$. We show that there are either trivial or torsion elements of order two or three.\newline

In order to compute the commutators of the loops around two distinct resonance hyperplanes, we first make a distinction between the mutual intersections between the corresponding partitions.
\begin{defn}\label{def:secpar}
For two nonempty strict subsets $I$ and $J$ of $\{1,\dots,p\}$, we consider the four mutual intersections $I \cap J$, $I^{\complement} \cap J$, $I \cap J^{\complement}$, $I^{\complement} \cap J^{\complement}$. If these four subsets are nonempty, we say that $I$ and $J$ are \textit{secant}. If among these four subsets, only three are nonempty, we say that $I$ and $J$ are \textit{parallel}. 
\end{defn}

Note that at least two  of the intersections in Lemma~\ref{def:secpar} are nonempty, and if precisely two are non empty, then the two partitions are the same.
In $\{1,2,3,4\}$,  the subsets $I=\{1,2\}$ and $J=\{1,3\}$ are secant and  $I=\{ 2 \}$ and $J=\{ 1,3 \}$ are parallel.

If the partitions corresponding to two resonance hyperplanes are secant, the monodromy action of each of the corresponding loops on any fiber commute.

\begin{lem}\label{lem:comute}
For any stratum $\mathcal{H}$, if $I$ and $J$ are secant, then the monodromy actions of~$\gamma_{I}$ and $\gamma_{J}$ on the isoresidual fibration commute.
\end{lem}

\begin{proof}
We consider a chamber $\mathcal{C}$ of real fibers adjacent to both hyperplanes $A_{I}$ and~$A_{J}$. To show that such a chamber exists, we choose a sign function $\psi$ such that $\psi(I \cap J)=\psi(I^{\complement} \cap J^{\complement})=+$, $\psi(I \cap J^{\complement})=\psi(I^{\complement} \cap J)=-$ and the sign of a subset belonging to an intersection is the same as this one. Then in the boundary of the corresponding chamber we can choose configuration $\lambda$ in such a way that the sum indexed by $I \cap J$ is the opposite of the sum index by $I \cap J^{\complement}$ and no other relations holds. This shows that $\mathcal{C}$ is adjacent to~$A_{I}$. One shows similarly that $\mathcal{C}$ is adjacent to~$A_{J}$. \newline
Now choose let us consider a fiber $\mathcal{F}$ in~$\mathcal{C}$.  We can show as in Propositions~\ref{prop:monodromieloop} and~\ref{prop:monodromieloopbis} that the monodromy action of $\gamma_{I}$ acts trivially on every element of $\mathcal{F}$ whose decorated tree does not have an edge $e_{1}$ corresponding to partition $I \cap I^{\complement}$. For~$\gamma_{J}$ the same holds with an edge~$e_{2}$. Since $I$ and $J$ are secant, there exist no $1$-form in $\mathcal{F}$ whose decorated tree has both edges $e_{1}$ and~$e_{2}$. Therefore, the loops $\gamma_{I}$ and $\gamma_{J}$ act nontrivially on disjoint subsets of~$\mathcal{F}$. Hence they commute.
\end{proof}

In contrast, the monodromy actions of two loops such that the partitions are parallel do not commute in general. We use the standard notation $[g,h]=g^{-1}h^{-1}gh$ for the commutator of $g$ and~$h$.
\newline

\begin{lem}\label{lem:threecycle}
Let $\mu = (a,-b_{1},\dots,-b_{p})$ be a partition of $-2$ such that $b_{1}=1$ and $I \sqcup J \sqcup K$  be a partition of $\{1,\dots,p\}$ such that $K=\{1\}$ 
and at least one of the following conditions holds:
\begin{enumerate}[(i)]
 \item $d_{I},d_{J} \geq 1$;
 \item $d_{I}=0$, $d_{J} \geq 2$  and $J$ is not a singleton (up to permute $I$ and $J$).
\end{enumerate}
Then the commutator $[\gamma_{I},\gamma_{J}]$ is a product of $\frac{d_{I}!\, d_{J}!}{(d_{I}+1-c_{I})!\, (d_{J}+1-c_{J})!}$ disjoint cycles of order three.
\end{lem}

\begin{proof}
The fact that $K=\{1\}$ and $b_{1}=1$ implies that  $d_{K}=0$. Analogously to the proof of Lemma~\ref{lem:comute}, we consider a chamber $\mathcal{C}$ of real fibers adjacent to two adequate resonance hyperplanes. We choose a sign function $\psi$ such that $\psi(K)= +$ and  $\psi(Y)=-$ for any nonempty subset $Y \subset I \cup J$. This chamber is then adjacent to $A_{I}$ and $A_{J}$.\newline
Let us consider a fiber $\mathcal{F}$ of~$\mathcal{C}$.
Following Propositions~\ref{prop:monodromieloop} and~\ref{prop:monodromieloopbis}, the permutation $\gamma_{I}$ acts nontrivially on an element of $\mathcal{F}$ if and only their decorated tree has an edge corresponding to partition $I \cup I^{\complement}$. This action preserves the shape of the each of the two subtrees  corresponding to $I$ and $I^{\complement}$,  but modify the way there are glued to~$K$. The same holds for $\gamma_{J}$.\newline
The $1$-forms in $\mathcal{F}$ on which both $\gamma_{I}$ and $\gamma_{J}$ act nontrivially are those whose decorated tree has two edges connecting the two subtrees $I$ and $J$ to the vertex corresponding to~$K$. The latter vertex is of weight one so it does not have any half-edge. We denote by $D$ the subset of~$\mathcal{F}$ on which both $\gamma_{I}$ and $\gamma_{J}$ act non trivially. A computation similar to the one done in Proposition~\ref{prop:distributiondeg} shows that the cardinal of~$D$ is $\frac{d_{I}!\, d_{J}!}{(d_{I}+1-c_{I})! \, (d_{J}+1-c_{J})!}$.\newline
 The fact that one of the two conditions (i) and (ii) holds implies that both $\gamma_{I}$ and $\gamma_{J}$ act on any element of $D$ as a cycle of length at least two. Any orbit of $\gamma_{I}$, resp. $\gamma_{J}$, that contains an element of $D$ then contains a unique element of~$D$.   Indeed, the subtree corresponding to $I$ is moved on some corner of $J$ and there is no edge connecting $J$ to $I \cup K$ in these trees.  Moreover, it is easy to check that there is precisely one element of $D$ in each cycle of $\gamma_{I}$ and $\gamma_{J}$. Consequently, the fiber $\mathcal{F}$ splits into a subset where $\gamma_{I}$ and $\gamma_{J}$ act trivially and a disjoint union of orbits of $\gamma_{I}$ and $\gamma_{J}$ intersecting at exactly one element of~$D$. We call these pairs of orbits of  $\gamma_{I}$ and $\gamma_{J}$ {\em eight-shaped}. We picture these orbits in Figure~\ref{fig:huit} in the case of the stratum $\mathcal{H}(4,-1,-2,-3)$. The general case is given by replacing the vertices labeled by $2$ and $3$ by the subtrees corresponding to $I$ and~$J$.\newline
 \begin{figure}[ht]
\begin{tikzpicture}[scale=1,decoration={
    markings,
    mark=at position 0.5 with {\arrow[very thick]{>}}}]
   
\node[circle,draw] (a) at  (-2,0) {$p_{1}$};
\node[circle,draw] (b) at  (0,0) {$p_{2}$};
\node[circle,draw] (c) at  (2,0) {$p_{3}$};
\draw [postaction={decorate}]  (b) --  (a);
\draw[postaction={decorate}] (c) -- (b);

\draw (c) -- ++(60:.6);
\draw(c) -- ++(-60:.6);
\draw (c) -- ++(120:.6);
\draw (c) -- ++(240:.6);
\draw(b) -- ++(-90:.6);
\draw (b) -- ++(90:.6);

\begin{scope}[yshift=-2cm]
\node[circle,draw] (a) at  (-2,0) {$p_{2}$};
\node[circle,draw] (b) at  (0,0) {$p_{1}$};
\node[circle,draw] (c) at  (2,0) {$p_{3}$};
\draw [postaction={decorate}]  (a) --  (b);
\draw[postaction={decorate}] (c) -- (b);

 \draw (a) -- ++(120:.6);
\draw (a) -- ++(240:.6);
\draw (c) -- ++(60:.6);
\draw(c) -- ++(-60:.6);
\draw (c) -- ++(120:.6);
\draw (c) -- ++(240:.6);

 \draw[->] (3,.1) .. controls ++(30:.4) and ++(-30:.4) .. (3,1.9) coordinate[pos=.5] (d);
 \node[right] at (d) {$\gamma_{3}$};
  \draw[->] (-3,1.9) .. controls ++(-150:.4) and ++(150:.4) .. (-3,.1) coordinate[pos=.5] (d);
 \node[left] at (d) {$\gamma_{3}$};
  \draw[->] (4,-1.3) .. controls ++(90:.4) and  ++(0:.4) .. (3,-.1) coordinate[pos=.5] (d);
 \node[right] at (d) {$\gamma_{2}$};
  \draw[->] (-3,-.1) .. controls ++(180:.4) and ++(90:.4) .. (-4,-1.3) coordinate[pos=.5] (d);
 \node[left] at (d) {$\gamma_{2}$};
 \end{scope}

\begin{scope}[yshift=-4cm]
\begin{scope}[xshift=-4cm]
\node[circle,draw] (a) at  (-2,0) {$p_{1}$};
\node[circle,draw] (b) at  (0,0) {$p_{3}$};
\node[circle,draw] (c) at  (2,0) {$p_{2}$};
\draw [postaction={decorate}]  (b) --  (a);
\draw[postaction={decorate}] (c) -- (b);

\draw (c) -- ++(60:.6);
\draw(c) -- ++(-60:.6);
\draw(b) -- ++(-90:.6);
\draw (b) -- ++(45:.6);
\draw (b) -- ++(90:.6);
\draw (b) -- ++(135:.6);

  \draw[->] (2.7,-.5) .. controls ++(-30:.4) and ++(-150:.4) .. (5.3,-.5) coordinate[pos=.5] (d);
 \node[above] at (d) {$\gamma_{2}$};
\end{scope}

\begin{scope}[xshift=4cm]
\node[circle,draw] (a) at  (-2,0) {$p_{1}$};
\node[circle,draw] (b) at  (0,0) {$p_{3}$};
\node[circle,draw] (c) at  (2,0) {$p_{2}$};
\draw [postaction={decorate}]  (b) --  (a);
\draw[postaction={decorate}] (c) -- (b);

\draw (c) -- ++(60:.6);
\draw(c) -- ++(-60:.6);
\draw(b) -- ++(90:.6);
\draw (b) -- ++(-45:.6);
\draw (b) -- ++(-90:.6);
\draw (b) -- ++(-135:.6);
\end{scope}
\end{scope}
 \end{tikzpicture}
 \caption{The action of $\gamma_{2}$ and $\gamma_{3}$ on the fiber above the chamber $\mathcal{C}$ for the stratum $\mathcal{H}(4,-1,-2,-3)$ (the arrows that are trivial are not represented).} \label{fig:huit}
\end{figure}
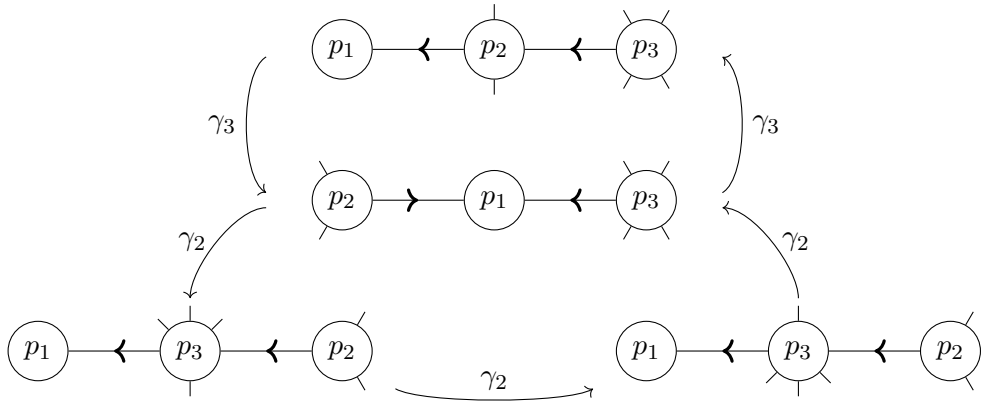
We consider a eight-shaped pair of orbits and let $t$ be its unique element that belongs to~$D$. Since the decorated tree $\gamma_{J}(t)$ is invariant by the action of $\gamma_{I}$,  we have the equality $[\gamma_{I},\gamma_{J}](t)=\gamma_{I}^{-1}(t)$.
Similarly, we get $[\gamma_{I},\gamma_{J}](\gamma_{I}^{-1}(t))=\gamma_{J}^{-1}(t)$ and
$[\gamma_{I},\gamma_{J}](\gamma_{J}^{-1}(t))=t$.
Therefore, the decorated trees $(t~~\gamma_{I}^{-1}(t)~~\gamma_{J}^{-1}(t))$ is a cycle of order three of the commutator of~$\gamma_{I}$ and~$\gamma_{J}$.\newline
For the elements of the eight-shaped pair of orbits different from~$t$, $\gamma_{I}^{-1}(t)$ and $\gamma_{J}^{-1}(t)$, neither the element nor its image by $\gamma_{I}$ or $\gamma_{J}$ belong to~$D$. Thus, the commutator simplifies into a product of two inverse elements and acts thus trivially on the fiber.
\end{proof}

Note that if both conditions (i) and (ii) of  Lemma~\ref{lem:threecycle} are not satisfied, then the monodromy action of one of the two loops is trivial (and would thus commute with any other element of the monodromy group), see Corollary \ref{cor:trivmon}.\newline

We still have to address the question of commutation relations in the case where $d_{K} \geq 1$. In that case, the commutators do not decompose into disjoint cycles of order three, but into an even number of disjoint transpositions.

\begin{lem}\label{lem:discomu}
Let $\mu = (a,-b_{1},\dots,-b_{p})$ be a partition of $-2$ and $I \sqcup J \sqcup K$  be a partition of $\{1,\dots,p\}$ such that  $d_{K} \geq 1$ and $d_{I}+d_{J} \geq 1$. Then the commutator $[\gamma_{I},\gamma_{J}]$ is an even product of disjoint transpositions.
\end{lem}

\begin{proof}
We consider a chamber $\mathcal{C}$ of real fibers such that two adequate resonance hyperplanes are incident to $\mathcal{C}$. We choose a sign function which is positive for any nonempty subset of $K$ and negative for any nonempty subset of $I \cup J$. Then, we choose a fiber $\mathcal{F}$ of~$\mathcal{C}$. This chamber is then adjacent to $A_{I}$ and $A_{J}$.\newline
Following Propositions \ref{prop:monodromieloop} and~\ref{prop:monodromieloopbis}, a permutation $\gamma_{I}$ acts nontrivially on an element of $\mathcal{F}$ if and only if its decorated tree has an edge corresponding to partition $I \cup I^{\complement}$. This action preserves the shape of the subtrees corresponding to $I$ and $I^{\complement}=J\cup K$, while moves the edge connecting these two subtrees. A similar description holds for~$\gamma_{J}$.\newline
The elements of $\mathcal{F}$ on which both $\gamma_{I}$ and $\gamma_{J}$ act nontrivially are those whose decorated tree has two edges connecting the two subtrees $I$ and $J$ to subtree $K$. We denote by~$D$ this subset of $\mathcal{F}$. A computation similar to that of Proposition \ref{prop:distributiondeg} shows that the cardinal of~$D$ is $\frac{d_{I}!\, d_{J}!\, (d_{K}+1)!}{(d_{I}+1-c_{I})!\, (d_{J}+1-c_{J})!\, (d_{K}+1-c_{K})!}$.\newline
The loop $\gamma_{I}$ acts by monodromy as a product of disjoint cycles of order $\lcm(d_{I},d_{J}+d_{K}+1)$ if $c_{I} \geq 2$ and of order $d_{J}+d_{K}+1$ if $I$ is a singleton. A similar statement holds for $\gamma_{J}$. In particular, both loops act nontrivially.\newline
In any nontrivial orbit of $\gamma_{I}$, the intersection with $D$ is either empty are formed by $d_{K}+1$ consecutive elements for the cyclic order. These elements are shared by exactly one nontrivial orbit of $\gamma_{J}$. Then, we essentially follow the end of the proof of Lemma \ref{lem:threecycle} studying the eight-shaped pair of orbits. The difference is that the intersection of the two orbits is not a singleton. For any element $t$ of the orbit of $\gamma_{I}$ such that neither $t$ nor $\gamma_{I}(t)$ belongs to~$D$, the action of the commutator simplifies into a product of two inverse permutations. The decorated tree $t$ is thus preserved by $[\gamma_{I},\gamma_{J}]$. The same holds for an element $t$ of the orbit of $\gamma_{J}$ such that neither $t$ nor $\gamma_{J}(t)$ belongs to $D$.\newline
For any element $t$ of $D$ such that both $\gamma_{I}(t)$ and $\gamma_{J}(t)$ belongs to~$D$, it is clear that the actions of these two permutations are inverse from each other. Consequently, $[\gamma_{I},\gamma_{J}]$ also preserves these decorated trees.\newline

Among the $d_{K}+1$ elements of $D$ belonging to the pair of orbits, we can see from the action of the decorated trees that an element $t$ such that $\gamma_{J}(t)$ is not in $D$ satisfies $t=\gamma_{I}(s)$ where $s$ is not in $D$. An analog result is true by transposing $I$ and $J$. This implies that $[\gamma_{I},\gamma_{J}]$ acts as a transposition for these two pairs of elements. Since $d_{K} \geq 1$, these two pairs are disjoint and we  get a product of permutations.
\end{proof}

The last cases missed by Lemma \ref{lem:discomu} are also those for which the monodromy action of one of the two loops would be trivial, see  Corollary \ref{cor:trivmon}.\newline

\subsection{The exceptional family of strata of $1$-forms with two simple poles}
\label{sec:monod2sim}

In this section, we treat the case of the strata of genus~$0$ of the form $\mathcal{H}(a,-kc_{1},\dots,-kc_{p-2},-1,-1)$ with $k \geq 2$, $a=k \sum_{j=1}^{p-2} c_{j}$ and $c_{1},\dots,c_{p-2}$ coprime. In this case, we assign a topological invariant that will break the monodromy group of the isoresidual fibration into a semidirect product.\newline

 In such a stratum, we consider a generic fiber $\mathcal{F}$ with real residues such that the residue of the first simple pole is negative while the residue of the second simple pole is positive.
For any translation surface defined by a meromorphic $1$-form $\omega$ of $\mathcal{F}$, we consider the class of oriented nonsingular paths that go from the first simple pole to the other and that are vertical outside a compact set. For any such path $\alpha$ we denote by
\begin{equation*}
 \theta(t)= \frac{ \omega( \alpha'(t))}{|\omega( \alpha'(t))|}
\end{equation*}  the Gauss map and by $\mathfrak{d}(\alpha)$ its degree.
Moreover, the degree of two such paths that differ by a simple loop around a singularity differ by the degree of this singularity. Therefore, the congruence class modulo $k$ of $\mathfrak{d}(\alpha)$ does not depends on $\alpha$. Hence, its class $\mathfrak{d}_{\omega}$ in $\mathbb{Z}_{k}$ is a topological invariant of the meromorphic $1$-form $\omega$ (and the associated translation structure) that we call the \textit{topological class} of~$\omega$.\newline

For such strata, the monodromy group $\mathcal{M}_{\mathcal{F}}$ of the isoresidual cover is not the full symmetric group $\mathfrak{S}(\mathcal{F})$.
\begin{prop}\label{prop:morphism}
For any stratum $\mathcal{H}(a,-kc_{1},\dots,-kc_{p-2},-1,-1)$ of genus zero such that $k \geq 2$, $a=k \sum_{j=1}^{p-2} c_{j}$ and $c_{1},\dots,c_{p-2}$ coprime, there is a surjective morphism from the monodromy group $\mathcal{M}_{\mathcal{F}}$  to $\mathbb{Z}_{k}$.
\end{prop}

\begin{proof}
It suffices to check that for any element $\omega$ of a generic isoresidual fiber~$\mathcal{F}$, the monodromy action of any loop around a hyperplane $A_{I}$ induces an action on the topological class $\mathfrak{d}_{\omega}$ of $\omega$ that is compatible with the group structure of $\mathbb{Z}_{k}$. Any resonance hyperplane corresponds to a nontrivial partition $I \cap I^{\complement}$ of the set of the poles. It is clear from Propositions~\ref{prop:monodromieloop} and~\ref{prop:monodromieloopbis} that the monodromy action of a loop around a hyperplane corresponding to a partition where the two simple poles are in the same subset preserves~$\mathfrak{d}_{\omega}$.\newline
On the opposite, the monodromy action of the loops $\gamma_{p-1}$ and $\gamma_{p}$ increases or decreases the topological class by one, i.e. $\mathfrak{d}_{\gamma_{p-1}(\omega)}=\mathfrak{d}_{\omega} \pm 1$. Indeed, this operation consists of rotating a simple pole by an angle of~$2\pi$.\newline
The last case is about resonance hyperplanes corresponding to partitions where none of the subset is a singleton and the two simple poles are not in the same subset. The geometric interpretation of the monodromy action of such loops (see the proof of Proposition~\ref{prop:monodromieloopbis})  is given by rotating the edge connecting around the two subtrees in the positive direction.  Therefore, the topological class~$\mathfrak{d}_{\omega}$ is also preserved in this case.\newline
Summing up, we conclude that the group morphism sending $\gamma_{I}$ to $0$ for every $I\neq \{p-1\},\{p\}$ and $\pm1$ to these two monodromies.
\end{proof}

Now recall that the generic fiber $\mathcal{F}$ of such stratum has $\frac{a!}{(a+2-p)!}$. From the previous property, we deduce the following useful result. 
\begin{cor}\label{cor:semidirect}
The monodromy group $\mathcal{M}_{\mathcal{F}}$ is the semidirect product of  $\mathbb{Z}_{k}$ and a subgroup of the symmetric group of a set of $\frac{a!}{k(a+2-p)!}$ elements.
\end{cor}

\begin{proof}
Since the monodromy group acts transitively on the topological classes of the elements of the fiber, the $\frac{a!}{(a+2-p)!}$ elements of the fiber are split into classes of $\frac{a!}{k(a+2-p)!}$ elements with the same topological class.
\end{proof}

Corollary \ref{cor:semidirect} also allows a direct computation of the monodromy group of a first family of examples.

\begin{cor}\label{cor:monoex}
For any $a \geq 1$, the monodromy group of the isoresidual fibration of $\mathcal{H}(a,-a,-1,-1)$ is the cyclic group $\mathbb{Z}_{a}$ of order~$a$.
\end{cor}

Note that in this example, the monodromy group of the isoresidual fibration is the smallest subgroup of the automorphism group of the generic fiber which is still transitive.\newline

\subsection{The case of strata of meromorphic $1$-forms with three poles}
\label{sec:monod3poles}

We consider a family of examples where the degrees of the singularities are slightly different from that of Corollary~\ref{cor:monoex} but where the monodromy group is far bigger. We recall first useful results about permutations, which are proved in~\cite{Pi} as Proposition~3 and Proposition~4.

\begin{lem}\label{lem:piccard}
Let $n\geq1$, the groups $\mathfrak{S}_{n}$ and $\mathfrak{A}_{n}$ be respectively the symmetric and alternating group of $n$ elements. Then the following two claims hold.
\begin{enumerate}[(i)]
 \item For any integer $1<k<n$ the permutations $(1~2~\dots~k)$ and $(1~2~\dots~n)$ generate $\mathfrak{A}_{n}$ if both $k$ and $n$ are odd and  $\mathfrak{S}_{n}$ if at least one of them is even.
 \item For any pair of integers $1<l\leq k<n$, the permutations $(1~2~\dots~k)$ and $(l~l+1~\dots~n)$ generate $\mathfrak{A}_{n}$ if both permutations are even and  $\mathfrak{S}_{n}$ if at least one permutation is odd, with the following three exceptions:
 $$(1~2~3~4),\, (3~4~5~6);\, (1~2~3~4),\,(2~3~4~5~6);\, (1~2~3~4~5),\, (3~4~5~6).$$
\end{enumerate}

\end{lem}

We now state a first generalisation of Corollary~\ref{cor:monoex}.
\begin{prop}\label{prop:monosimple}
For any $s,t \geq 2$, the monodromy group of the isoresidual fibration of the stratum $\mathcal{H}(s+t-1,-s,-t,-1)$ is:
\begin{enumerate}[(i)]
 \item  the alternating group $\mathfrak{A}(\mathcal{F})$ if both $s$ and $t$ are odd;
 \item the full symmetric group $\mathfrak{S}(\mathcal{F})$ otherwise.
\end{enumerate}
\end{prop}

\begin{proof}
According to point (i) of Corollary \ref{cor:alternate}, the monodromy group of the isoresidual fibration of such a stratum is a subgroup of the alternating group $\mathfrak{A}(\mathcal{F})$ if and only if both~$s$ and~$t$ are odd. Therefore, it suffices to consider the case where $s$ and $t$ are odd and in that case to find a set of elements of the monodromy group that generates the whole alternating group~$\mathfrak{A}(\mathcal{F})$.\newline

According to Theorem~\ref{thm:numerofibre}, the generic fibers have $s+t-1$ elements. The action by monodromy of the loops around each resonance hyperplane can be decomposed into cycles. It follows from Proposition~\ref{prop:monodromieloop} that the simple loop $\gamma_{3}$ around the hyperplane $A_{3}$ acts as a cycle of order $s+t-1$.  In contrast, the simple loop $\gamma_{1}$ around $A_{1}$ acts as a cycle of order~$t$ and the cycle $\gamma_{2}$ around $A_{1}$ acts as a cycle of order~$s$.\newline
We consider a fiber $\mathcal{F}$ above the chamber where $\lambda_{1},\lambda_{3}<0$. The elements of $\mathcal{F}$ are classified by their decorated trees. These trees are rooted trees where the root is the vertex of weight~$t$. Among the $s+t-1$ decorated trees corresponding to elements of $\mathcal{F}$, there are~$t$ of them such that  the vertices of weight $1$ and $s$ are directly connected to the root. The $s-1$ remaining decorated trees are such that the vertex connected to the root has weight~$s$ while the vertex of weight $1$ is connected to one of the $s-1$ allowed corners of the latter vertex.\newline

The $t$ trees where both vertices are connected to the root are cyclically permuted by the monodromy action of~$\gamma_{1}$. Moreover $\gamma_{1}$ fixes the other elements. The action of~$\gamma_{3}$ moves cyclically the vertex of weight $1$ around the rest of the tree. Hence this action form a cycle containing all the decorated trees. Consequently, the monodromy action of the loops around $A_{1}$ and $A_{3}$ form a pair of elements of $\mathfrak{S}(\mathcal{F})$ conjugated to the permutations $(1~2~\dots~t)$ and $(1~2~\dots~s+t-1)$. Now it suffices to use point (i) of Lemma~\ref{lem:piccard} to conclude the proof.
\end{proof}

Generalizing Proposition~\ref{prop:monosimple}, we are able to compute the monodromy group of the isoresidual cover for any stratum with three poles. In the following proposition, we state all the cases but one which will be treated in Proposition~\ref{prop:last}.

\begin{prop}\label{prop:monocomp}
For any stratum $\mathcal{H}(a,-b_{1},-b_{2},-b_{3})$ such that $b_{1},b_{2},b_{3} \geq 2$ which is different from $\mathcal{H}(6,-2,-3,-3)$, the monodromy group of the isoresidual cover is isomorphic to the alternating group $\mathfrak{A}_{a}$ if $b_{1},b_{2},b_{3}$ have the same parity and symmetric group $\mathfrak{S}_{a}$ otherwise.
\end{prop}

\begin{proof}
For these strata, the generic fibers of the isoresidual fibration have~$a$ elements. According to Proposition~\ref{prop:monodromieloop}, the loop $\gamma_{i}$ acts as a cycle of order $a+1-b_{i}$. In particular each of these cycles is of length at least three. Two of these cycles act transitively on the fiber. Indeed, consider a fiber in the chamber bounded by $A_{i}$ and $A_{j}$. Then one of the two edges of any decorated tree in the fiber separate either the vertex $i$ or the vertex $j$ from the rest of the tree. The intersection of two of these cycles is formed by $a+2-b_{i}-b_{j} \geq 2$ consecutive elements for each of the two cycles. Therefore this pair of permutations is conjugate to the pair of permutations $\alpha=(1~2~\dots~s)$ and $\beta=(t~t+1~\dots~a)$ with $s=a+1-b_{i}$ and $t=b_{j}$  in $\mathfrak{S}_{a}$. In particular, we have $s\geq t+1$, $s \geq 3$ and $a-t\geq 2$.\newline
Now the point (ii) of Lemma~\ref{lem:piccard} implies that such pairs generate~$\mathfrak{A}_{a}$ with the three exceptions where
$\alpha=(1~2~3~4)$ and $\beta=(3~4~5~6)$, or $\alpha=(1~2~3~4)$ and $\beta=(2~3~4~5~6)$, or $\alpha=(1~2~3~4~5)$ and $\beta=(3~4~5~6)$. The only stratum with $p=3$ such that any pair of generators is one of these exceptions is $\mathcal{H}(6,-2,-3,-3)$. 
\end{proof}

Finally, we treat the case of the stratum $\mathcal{H}(6,-2,-3,-3)$. Recall that the action of $\mathfrak{S}_{5}$ by conjugation on its $5$-Sylow gives an embedding of $\mathfrak{S}_{5}$ into $\mathfrak{S}_{6}$, that is called the {\em exotic embedding}.

\begin{prop}\label{prop:last}
The monodromy group of the isoresidual cover of $\mathcal{H}(6,-2,-3,-3)$ is isomorphic to $\mathfrak{S}_{5}$ and is embedded into the automorphism group of any generic fiber as the exotic embedding of $\mathfrak{S}_{5}$ into $\mathfrak{S}_{6}$.
\end{prop}

\begin{proof}
According to Theorem~\ref{thm:numerofibre}, the generic fiber of the isoresidual cover of the stratum $\mathcal{H}(6,-2,-3,-3)$ has $6$ elements. In the chamber $\mathcal{C}$ where $\lambda_{3}>0$ while $\lambda_{1},\lambda_{2}<0$, these elements are classified by decorated trees where the vertex labeled by $3$ is the root. We denote by $A$ the unique tree such that vertex $2$ is glued on vertex $1$ which is glued on vertex $3$. Then, we denote by $B$, $C$ and $D$ the three trees where the vertices $1$ and $2$ are directly glued on the root, where the number of half-edges on vertex $3$ between the vertices~$1$ and $2$  is respectively $0$, $2$ and $4$ in the clockwise order.  The last two trees are such that vertex~$1$ is glued on vertex $2$ which is glued on vertex $3$. We denote these trees by $E$ or $F$ depending on if the number of half edges at vertex $2$ between vertices $1$ and $3$ is $1$ or $3$ in the clockwise order. These decorated trees are picture in Figure~\ref{fig:dectreepreuve}.\newline
\begin{figure}[ht]
\begin{tikzpicture}[scale=1,decoration={
    markings,
    mark=at position 0.5 with {\arrow[very thick]{>}}}]
   
   \node at (-3,0) {$A$:};
\node[circle,draw] (a) at  (-2,0) {$p_{2}$};
\node[circle,draw] (b) at  (0,0) {$p_{1}$};
\node[circle,draw] (c) at  (2,0) {$p_{3}$};
\draw [postaction={decorate}]  (a) --  (b);
\draw[postaction={decorate}] (b) -- (c);

\draw (a) -- ++(120:.6);
\draw (a) -- ++(240:.6);
\draw (a) -- ++(60:.6);
\draw (a) -- ++(-60:.6);
\draw (c) -- ++(60:.6);
\draw(c) -- ++(-60:.6);
\draw (c) -- ++(120:.6);
\draw (c) -- ++(240:.6);
\draw(b) -- ++(-90:.6);
\draw (b) -- ++(90:.6);

\begin{scope}[xshift=7cm]
  \node at (-3,0) {$B$:};
\node[circle,draw] (a) at  (-2,0) {$p_{1}$};
\node[circle,draw] (b) at  (0,0) {$p_{3}$};
\node[circle,draw] (c) at  (2,0) {$p_{2}$};
\draw [postaction={decorate}]  (a) --  (b);
\draw[postaction={decorate}] (c) -- (b);

 \draw (a) -- ++(120:.6);
\draw (a) -- ++(240:.6);
\draw (c) -- ++(60:.6);
\draw(c) -- ++(-60:.6);
\draw (c) -- ++(120:.6);
\draw (c) -- ++(240:.6);
\draw(b) -- ++(40:.6);
\draw (b) -- ++(140:.6);
\draw (b) -- ++(70:.6);
\draw (b) -- ++(110:.6);
\end{scope}

\begin{scope}[yshift=-2cm]
  \node at (-3,0) {$C$:};
\node[circle,draw] (a) at  (-2,0) {$p_{1}$};
\node[circle,draw] (b) at  (0,0) {$p_{3}$};
\node[circle,draw] (c) at  (2,0) {$p_{2}$};
\draw [postaction={decorate}]  (a) --  (b);
\draw[postaction={decorate}] (c) -- (b);

 \draw (a) -- ++(120:.6);
\draw (a) -- ++(240:.6);
\draw (c) -- ++(60:.6);
\draw(c) -- ++(-60:.6);
\draw (c) -- ++(120:.6);
\draw (c) -- ++(240:.6);
\draw(b) -- ++(-70:.6);
\draw (b) -- ++(110:.6);
\draw (b) -- ++(70:.6);
\draw (b) -- ++(-110:.6);

\begin{scope}[xshift=7cm]
  \node at (-3,0) {$D$:};
\node[circle,draw] (a) at  (-2,0) {$p_{1}$};
\node[circle,draw] (b) at  (0,0) {$p_{3}$};
\node[circle,draw] (c) at  (2,0) {$p_{2}$};
\draw [postaction={decorate}]  (a) --  (b);
\draw[postaction={decorate}] (c) -- (b);

 \draw (a) -- ++(120:.6);
\draw (a) -- ++(240:.6);
\draw (c) -- ++(60:.6);
\draw(c) -- ++(-60:.6);
\draw (c) -- ++(120:.6);
\draw (c) -- ++(240:.6);
\draw(b) -- ++(-40:.6);
\draw (b) -- ++(-140:.6);
\draw (b) -- ++(-70:.6);
\draw (b) -- ++(-110:.6);
\end{scope}
\end{scope}

\begin{scope}[yshift=-4cm]
  \node at (-3,0) {$E$:};
\node[circle,draw] (a) at  (-2,0) {$p_{1}$};
\node[circle,draw] (b) at  (0,0) {$p_{2}$};
\node[circle,draw] (c) at  (2,0) {$p_{3}$};
\draw [postaction={decorate}]  (a) --  (b);
\draw[postaction={decorate}] (b) -- (c);

 \draw (a) -- ++(120:.6);
\draw (a) -- ++(240:.6);
\draw (c) -- ++(60:.6);
\draw(c) -- ++(-60:.6);
\draw (c) -- ++(120:.6);
\draw (c) -- ++(240:.6);
\draw(b) -- ++(-90:.6);
\draw (b) -- ++(45:.6);
\draw (b) -- ++(90:.6);
\draw (b) -- ++(135:.6);

\begin{scope}[xshift=7cm]
  \node at (-3,0) {$F$:};
\node[circle,draw] (a) at  (-2,0) {$p_{1}$};
\node[circle,draw] (b) at  (0,0) {$p_{2}$};
\node[circle,draw] (c) at  (2,0) {$p_{3}$};
\draw [postaction={decorate}]  (a) --  (b);
\draw[postaction={decorate}] (b) -- (c);

 \draw (a) -- ++(120:.6);
\draw (a) -- ++(240:.6);
\draw (c) -- ++(60:.6);
\draw(c) -- ++(-60:.6);
\draw (c) -- ++(120:.6);
\draw (c) -- ++(240:.6);
\draw(b) -- ++(90:.6);
\draw (b) -- ++(-45:.6);
\draw (b) -- ++(-90:.6);
\draw (b) -- ++(-135:.6);
\end{scope}
\end{scope}
 \end{tikzpicture}
 \caption{The decorated trees corresponding to the $1$-forms in $\mathcal{H}(6,-2,-3,-3)$ above the chamber $\mathcal{C}$.} \label{fig:dectreepreuve}
\end{figure}
From this description, it is clear that the monodromy action of $\gamma_{2}$ is the cycle $(A~B~C~D)$ and the 
one of $\gamma_{1}^{-1}$ is $(B~C~D~E~F)$.\newline

We also consider chamber $\mathcal{C'}$ where $\lambda_{1},\lambda_{3}>0$ and $\lambda_{2}<0$. The chambers $\mathcal{C}$ and $\mathcal{C'}$ are separated by the resonance hyperplane $A_{1}$. In order to  compute the monodromy action of $\gamma_{3}$ on $\{A,B,C,D,E,F\}$, we identify the elements of fibers of these two chambers in the following way. The decorated trees associated to the $1$-forms above $\mathcal{C}'$ are pictured in Figure~\ref{fig:dectreepreuve2}.  \begin{figure}[ht]
\begin{tikzpicture}[scale=1,decoration={
    markings,
    mark=at position 0.5 with {\arrow[very thick]{>}}}]
   
   \node at (-3,0) {$A'$:};
\node[circle,draw] (a) at  (-2,0) {$p_{2}$};
\node[circle,draw] (b) at  (0,0) {$p_{1}$};
\node[circle,draw] (c) at  (2,0) {$p_{3}$};
\draw [postaction={decorate}]  (a) --  (b);
\draw[postaction={decorate}] (b) -- (c);

\draw (a) -- ++(120:.6);
\draw (a) -- ++(240:.6);
\draw (a) -- ++(60:.6);
\draw (a) -- ++(-60:.6);
\draw (c) -- ++(60:.6);
\draw(c) -- ++(-60:.6);
\draw (c) -- ++(120:.6);
\draw (c) -- ++(240:.6);
\draw(b) -- ++(-90:.6);
\draw (b) -- ++(90:.6);

\begin{scope}[xshift=7cm]
  \node at (-3,0) {$B'$:};
\node[circle,draw] (a) at  (-2,0) {$p_{1}$};
\node[circle,draw] (b) at  (0,0) {$p_{2}$};
\node[circle,draw] (c) at  (2,0) {$p_{3}$};
\draw [postaction={decorate}]  (b) --  (a);
\draw[postaction={decorate}] (b) -- (c);

 \draw (a) -- ++(120:.6);
\draw (a) -- ++(240:.6);
\draw (c) -- ++(60:.6);
\draw(c) -- ++(-60:.6);
\draw (c) -- ++(120:.6);
\draw (c) -- ++(240:.6);
\draw(b) -- ++(40:.6);
\draw (b) -- ++(140:.6);
\draw (b) -- ++(70:.6);
\draw (b) -- ++(110:.6);
\end{scope}

\begin{scope}[yshift=-2cm]
  \node at (-3,0) {$C'$:};
\node[circle,draw] (a) at  (-2,0) {$p_{1}$};
\node[circle,draw] (b) at  (0,0) {$p_{2}$};
\node[circle,draw] (c) at  (2,0) {$p_{3}$};
\draw [postaction={decorate}]  (b) --  (a);
\draw[postaction={decorate}] (b) -- (c);

 \draw (a) -- ++(120:.6);
\draw (a) -- ++(240:.6);
\draw (c) -- ++(60:.6);
\draw(c) -- ++(-60:.6);
\draw (c) -- ++(120:.6);
\draw (c) -- ++(240:.6);
\draw(b) -- ++(-70:.6);
\draw (b) -- ++(110:.6);
\draw (b) -- ++(70:.6);
\draw (b) -- ++(-110:.6);

\begin{scope}[xshift=7cm]
  \node at (-3,0) {$D'$:};
\node[circle,draw] (a) at  (-2,0) {$p_{1}$};
\node[circle,draw] (b) at  (0,0) {$p_{2}$};
\node[circle,draw] (c) at  (2,0) {$p_{3}$};
\draw [postaction={decorate}]  (b) --  (a);
\draw[postaction={decorate}] (b) -- (c);

 \draw (a) -- ++(120:.6);
\draw (a) -- ++(240:.6);
\draw (c) -- ++(60:.6);
\draw(c) -- ++(-60:.6);
\draw (c) -- ++(120:.6);
\draw (c) -- ++(240:.6);
\draw(b) -- ++(-40:.6);
\draw (b) -- ++(-140:.6);
\draw (b) -- ++(-70:.6);
\draw (b) -- ++(-110:.6);
\end{scope}
\end{scope}

\begin{scope}[yshift=-4cm]
  \node at (-3,0) {$E'$:};
\node[circle,draw] (a) at  (-2,0) {$p_{1}$};
\node[circle,draw] (b) at  (0,0) {$p_{3}$};
\node[circle,draw] (c) at  (2,0) {$p_{2}$};
\draw [postaction={decorate}]  (b) --  (a);
\draw[postaction={decorate}] (c) -- (b);

 \draw (a) -- ++(120:.6);
\draw (a) -- ++(240:.6);
\draw (c) -- ++(60:.6);
\draw(c) -- ++(-60:.6);
\draw (c) -- ++(120:.6);
\draw (c) -- ++(240:.6);
\draw(b) -- ++(-90:.6);
\draw (b) -- ++(45:.6);
\draw (b) -- ++(90:.6);
\draw (b) -- ++(135:.6);

\begin{scope}[xshift=7cm]
  \node at (-3,0) {$F'$:};
\node[circle,draw] (a) at  (-2,0) {$p_{1}$};
\node[circle,draw] (b) at  (0,0) {$p_{3}$};
\node[circle,draw] (c) at  (2,0) {$p_{2}$};
\draw [postaction={decorate}]  (b) --  (a);
\draw[postaction={decorate}] (c) -- (b);

 \draw (a) -- ++(120:.6);
\draw (a) -- ++(240:.6);
\draw (c) -- ++(60:.6);
\draw(c) -- ++(-60:.6);
\draw (c) -- ++(120:.6);
\draw (c) -- ++(240:.6);
\draw(b) -- ++(90:.6);
\draw (b) -- ++(-45:.6);
\draw (b) -- ++(-90:.6);
\draw (b) -- ++(-135:.6);
\end{scope}
\end{scope}
 \end{tikzpicture}
 \caption{The decorated trees corresponding to the $1$-forms in $\mathcal{H}(6,-2,-3,-3)$ above the chamber $\mathcal{C}'$.} \label{fig:dectreepreuve2}
\end{figure}  We consider the path $\sigma_{1}$ going from $\mathcal{C}$ to $\mathcal{C}'$ and which coincide with $\gamma_{1}$ where there are both define. 
The action of $\sigma_{1}$ on the decorated trees is given by the following permutation
\[ \left(
  \begin{array}{ c c c c c c c}
     A & B & C & D & E & F \\
    A' & D' & E' & F' & B' & C'
  \end{array} \right).
\]
The monodromy action of $\gamma_{3}$ on the fiber above $\mathcal{C}$ is then computed using the path $\gamma_{3}'=\sigma_{1}\gamma_{3}\sigma_{1}^{-1}$.\newline
   The monodromy of $\gamma_{3}$ on the fiber above $\mathcal{C}'$ is the cycle $(A'~B'~C'~D')$. Hence the monodromy of $\gamma_{3}$ on the fiber above $\mathcal{C}$ is the cycle $(A~E~F~B)$.\newline 

The group generated by the permutations  $(A~B~C~D)$, $(B~C~D~E~F)$ and $(A~E~F~B)$ has order $120$ and acts transitively on the set of six elements. This the exotic embedding of $\mathfrak{S}_{5}$ into~$\mathfrak{S}_{6}$.
\end{proof}

Finally note that Theorem~\ref{thm:monothree} is proved by Corollary~\ref{cor:semidirect}, Propositions~\ref{prop:monosimple},~\ref{prop:monocomp} and~\ref{prop:last} . Together they give a complete description of the monodromy groups of the isoresidual cover for strata with three poles. Conjecturally, we expect that for $p \geq 4$, the monodromy group is as big as possible under the constraints of Corollaries~\ref{cor:alternate} and~\ref{cor:semidirect}.\newline

\paragraph{\bf Acknowledgements.} The first author is supported by the  project CONACyT A1-S-9029 "Moduli de curvas y Curvatura en $A_{g}$" of Abel Castorena.  The second named author is grateful to Ben-Michael Kohli, Carlos Matheus, Dmitry Novikov, Boris Shapiro for their valuable remarks.
The second author is supported by the Israel Science Foundation (grant No. 1167/17) and the European Research Council (ERC) under the European Union Horizon 2020 research and innovation programme (grant agreement No. 802107).\newline

\paragraph{\bf Conflicts of interest.} On behalf of all authors, the corresponding author states that there is no conflict of interest.

\nopagebreak
\vskip.5cm
\end{document}